\documentclass{article}

\usepackage{amssymb}
\usepackage{amsthm}
\usepackage{amsmath}
\usepackage{authblk}
\usepackage{bm}
\usepackage{tikz}
\usepackage{lineno}
\usetikzlibrary{arrows.meta}
\usetikzlibrary{arrows}
\newtheorem{theorem}{Theorem}
\newtheorem{lemma}[theorem]{Lemma}
\newtheorem{remark}{Remark}
\newtheorem{definition}[theorem]{Definition}
\newtheorem{assumption}{Assumption}
\newtheorem{example}{Example}

\DeclareMathOperator{\Ima}{Im}
\DeclareMathOperator{\Id}{Id}

\newcommand{\red}[1]{{\color{red} #1}}

\DeclareMathOperator{\diver}{div}
\newcommand{\grad}{\textbf{\textup{grad}}}
\newcommand{\tensor}[1]{\overline{\bm{#1}}}
\newcommand{\Diver}{\textbf{\textup{Div}}}
\newcommand{\Grad}{\textbf{\textup{Grad}}}
\newcommand{\curldeuxD}{\text{curl}_{2D}}
\newcommand{\curl}{\textbf{\textup{curl}}}
\newcommand{\perpext}{{\perp \!\!\!\! \perp_{-}}}
\renewcommand{\d}{\textup{d}}
\newcommand{\dsp}{\displaystyle}

\title{Stokes-Lagrange and Stokes-Dirac representations of $N$-dimensional port-Hamiltonian systems for modelling and control\footnote{This work was supported by the IMPACTS project funded by the French National Research Agency (project ANR-21-CE48-001), https://impacts.ens2m.fr/}}

\author[1]{Antoine Bendimerad-Hohl}
\author[1]{Ghislain Haine}
\author[2]{Laurent Lef\`{e}vre}
\author[1]{Denis Matignon}

\affil[1]{Fédération ENAC ISAE-SUPAERO ONERA, Université de Toulouse, Toulouse, France}
\affil[2]{Univ. Grenoble Alpes, Grenoble INP, LCIS, Valence, France}

\begin{document}

\maketitle

\begin{abstract}
In this paper, we extend the port-Hamiltonian framework by introducing the concept of Stokes-Lagrange structure, which enables the implicit definition of a Hamiltonian over an $N$-dimensional domain and incorporates energy ports into the system. This new framework parallels the existing Dirac and Stokes-Dirac structures. We propose the Stokes-Lagrange structure as a specific case where the subspace is explicitly defined via differential operators that satisfy an integration by parts formula. By examining various examples through the lens of the Stokes-Lagrange structure, we demonstrate the existence of multiple equivalent system representations. These representations provide significant advantages for both numerical simulation and control design, offering additional tools for the modelling and control of port-Hamiltonian systems.\\
\newline
\textbf{Keywords:} Port-Hamiltonian systems; Stokes-Dirac structure; Stokes-Lagrange structure; Reissner-Mindlin model; Kirchhoff-Love model; \linebreak{}Maxwell's equations; Passivity-based control
\newline
\textbf{Mathematics Subject Classification:} 93C20; 37K06; 74K20; 35Q61
\end{abstract}

\section{Introduction}
\label{sec:intro}

The Port-Hamiltonian (pH) framework stems from classical Hamiltonian formalism and bond graphs theory. As in the classical Hamiltonian formalism, the dynamics of the state $\alpha$ derive from a Hamiltonian functional $H(\alpha)$ (\textit{e.g.} energy, entropy, exergy, \textit{etc.}), however the underlying structure is not symplectic but defined as a more general \emph{Dirac} structure (in particular the dimension of the state space is not necessarily even). This structure,  similarly to bond graphs, allows for representing systems as subsystems interconnected through power preserving interconnections. In particular, it is composed of pairs of flow and effort variables $(f,e)$ named \emph{power ports} whose duality product $\langle e,f\rangle$ corresponds to a power; the Dirac structure then guarantees that the power balance over its power ports is zero: $\sum_i \langle e_i,f_i\rangle = 0$ which enforces the preservation of energy.

These ports are usually of three different natures: storage ports, control ports and dissipative ports. Storage ports contain a state variable $\alpha$ and its flows and effort are defined as $f_s = - \frac{\d}{\d t}\alpha,\, e_s = \nabla_\alpha H$. Control ports allow for both controlling the system and interconnecting it with another pH system, its flows and effort are usually defined as the control and collocated observation respectively $f=u, \, e=y$ (note that ensuring that $\langle e_u, f_u\rangle$ is well-defined and  corresponds to a power requires the observation to be collocated with the control). Finally, resistive ports allow for modeling dissipativity, the flow and effort corresponding (in the mechanical case) to some velocity and friction force respectively such that $f_r= - R\, e_r$ with $R \geq 0$. Writing the power balance $\sum_i \langle e_i,f_i\rangle = 0$ gives us that $\frac{\d}{\d t}H(\alpha) = \langle u,y \rangle - \langle e_r, R e_r\rangle \leq \langle u,y \rangle $; the variation of energy is equal to the power provided by the control minus the dissipative phenomena.

PH systems have been developed to model, simulate, and control complex multiphysics systems, including mechanics, electromagnetism, and irreversible thermodynamics (see~\cite{van2014port} for an introductory textbook and~\cite{duindam2009modeling} for a wide array of applications, and more specifically~\cite{Gay_Balmaz_2017} for nonequilibrium thermodynamics using a Lagrangian variational formulation). The seminal work of~\cite{van2002hamiltonian} extended the pH framework to distributed parameter systems (DPS) with boundary energy flows. In particular, the semigroup approach applied to linear distributed parameter pH systems has yielded elegant results in the 1D case (see~\cite{jacob2012linear} for an introduction), with recent extensions to $N$-dimensional systems presented in~\cite{brugnoli2023stokes}. Since 2002, the body of literature on distributed pH systems has grown significantly, encompassing both theoretical developments and application-based studies (see~\cite{rashad2020twenty} for a comprehensive review).


PH systems with algebraic constraints have also been explored, leading to the development of finite-dimensional pH Differential-Algebraic Equation (pH-DAE, see for instance~\cite{beattie2018linear} for linear descriptor systems in matrix representation,~\cite{Zwart_2024} for abstract dissipative Hamiltonian DAE in infinite dimension, and~\cite{mehrmann2023control} for a review on control applications). These arise naturally from the interconnection of subsystems, which results in constraints between effort and flow variables in the underlying Dirac structure. Alternatively, constraints may stem from implicit energy constitutive equations, affecting the relationship between effort and energy state variables in a Lagrangian submanifold (see~\cite{van2018generalized} and~\cite{van2020dirac} for nonlinear cases).


Recently, there has been increasing interest in distributed parameter models represented in implicit form within the pH framework. These models include examples such as the implicit formulation of the Allen-Cahn equation~\cite{yaghi2022port}, the Dzektser equation governing the seepage of underground water~\cite{jacob2022solvability}, and the nonlocal visco-elastic models for nanorods~\cite{heidari2019port,heidari2022nonlocal}. Accompanying these advances, structure-preserving numerical methods have been proposed (see~\cite{bendimerad2023implicit}). The introduction of boundary control pH systems has been extended to systems where the variational derivative of the Hamiltonian is replaced by reciprocal operators, as shown in~\cite{maschke2023linear}. This generalization has led to the definition of a boundary pH system on a Stokes-Lagrange subspace, allowing the representation of previously discussed implicit models, such as the elastic rod with non-local elasticity.

When working with distributed pH systems governed by Partial Differential Equations (PDEs) on a spatial domain, special attention must be given to boundary variables. Specifically, Stokes-Dirac structures are employed to account for boundary power ports, while Stokes-Lagrange structures are used to describe boundary energy ports (see~\cite{krhac2024port} and~\cite{schaft2024boundary}). The explicit formulation of these structures plays a crucial role in ensuring energy-conserving behavior within the system dynamics. Field theory has been used to study different formulations of port-Hamiltonian systems; see \cite{schoberl2011first} for an introduction. This approach allows for different representations of pH systems \cite{SchoberlSiukaECC2013,scholberl2014auto} by considering different definitions of the energy variables.


This paper aims to extend the preliminary work presented in~\cite{bendimerad2024implicit} and achieve three objectives. Firstly, it defines the Stokes-Lagrange structure by following the methodology of~\cite{kurula2010dirac} on Dirac structures and  extending the Stokes-Lagrange subspaces presented in \cite{maschke2023linear} to $N$-dimensional domains. Secondly, it presents examples of 2D and 3D systems that admit both Stokes-Dirac and Stokes-Lagrange representations; \textit{i.e.} systems that can be written using a Stokes-Dirac or a Stokes-Lagrange structure. Thirdly, it discusses the differences and relations between these two structures and how the newly defined Stokes-Lagrange structure can be useful for modelling, numerics and control. In particular, it is shown how the Stokes-Lagrange structure allows  for alternate control laws using \emph{energy} control port instead of \emph{power} control port.

The paper is organized as follows: in Section~\ref{sec:intro}, the main ideas of the article are presented using a simple model, namely the 1D linear wave equation; in particular how different representations are obtained considering either Stokes-Dirac or Stokes-Lagrange representation and how implicit constitutive relations can be taken into account. In Section~\ref{sec:Stokes-Lagrange}, the definitions of a Stokes-Dirac and a Stokes-Lagrange structure are presented. These structures and their relations are then illustrated in Section~\ref{sec:examples} through a series of examples: the Reissner-Mindlin thick plate model and the Kirchhoff-Love thin plate model; the Maxwell's equation in classical and vector potential formulation and finally a dissipative nonlocal case,  the Dzektser equation. For the plate models and Maxwell's equations, two representations are presented using either Stokes-Dirac or Stokes-Lagrange structures.
In Section~\ref{sec:control}, applications to boundary and distributed control using \emph{energy} control ports are presented; in particular it is shown how the energy control ports allows for interconnections using nonseparable Hamiltonian functionals.
Finally, 5 appendices are to be found at the end of this paper which detail more precisely the operator transposition procedure~\ref{apx:operator-transposition},  Maxwell reciprocity conditions~\ref{apx:maxwell-reciprocity}, the Stokes-Lagrange structure~\ref{apx:stokes-lagrange} and the last two contain proofs of variational derivatives~\ref{apx:var-deriv} and Stokes' identities~\ref{apx:stokes-identities}.

\subsection{Study of a simple case: the 1D wave equation}

Let us begin with the 1D linear wave equation. This equation has already been thoroughly studied in the pH community~\cite{van2002hamiltonian,van2014port,haine2023numerical,maschke2023linear}, its relative simplicity will allow us to introduce the main ideas of the paper.

Let $[a,b] \subset \mathbb{R}$ be an interval, and consider $w(t,x)$ the vertical displacement of a string over this interval, then the wave equation reads~\cite{van2002hamiltonian}:
\begin{equation} \label{eq:1D-Wave}
    \partial_t (\rho \, \partial_t w) = \partial_x (E \, \partial_x w),
\end{equation}
with $\rho(x)>0$ the density and $E(x)>0$ the tension of the string. A Hamiltonian associated to this equation is the total mechanical energy given by the sum of the kinetic and elastic energies:
\begin{equation} \label{eqn:ham-wave}
    H(w) := \frac{1}{2}  \int_a^b \rho (\partial_t w)^2 + E (\partial_x w)^2 \, \d x \, .
\end{equation}

\subsubsection{The classical pH representation}

To write~\eqref{eq:1D-Wave} as a pH system, let us follow~\cite{van2002hamiltonian}. Choose \emph{energy variables} to write the Hamiltonian: $p := \rho \, \partial_t w$ the linear momentum, and $\varepsilon := \partial_x w$ the strain. Then the Hamiltonian~\eqref{eqn:ham-wave} becomes\footnote{The upper-scripts $\cdot^{SD}$ and $\cdot^{SL}$ refer to the Stokes-Dirac or Stokes-Lagrange representation of the system. The sub-script $\cdot_W$ refers to the Wave equation.}: 
\begin{equation} \label{eqn:ham-wave-explicit}
    H_W^{SD}(\varepsilon, p) = \frac{1}{2}  \int_a^b \frac{1}{\rho}\, p^2 + E\, \varepsilon^2 \, \d x \, .
\end{equation}
The \emph{co-state} variables are then computed as the variational derivatives~\cite{olver1993applications} of the Hamiltonian~\eqref{eqn:ham-wave-explicit} with respect to the state variables $\varepsilon, p$, which yields: $v := \delta_p H_W^{SD} = \frac{1}{\rho} \, p$ the velocity, and $\sigma = \delta_\varepsilon H_W^{SD} = E \, \varepsilon$ the stress. Then, the classical pH formulation of the 1D linear wave equation~\eqref{eq:1D-Wave} reads:
\begin{equation} \label{eqn:wave-equation-explicit-1}
    \frac{\partial}{\partial t}\begin{bmatrix}
        \varepsilon \\
        p
    \end{bmatrix} = \begin{bmatrix}
        0 & \partial_x \\ \partial_x & 0
    \end{bmatrix} \begin{bmatrix}
        \sigma \\ v
    \end{bmatrix},
\end{equation}
together with the constitutive relations, namely Hooke's law and the definition of linear momentum:
\begin{equation} \label{eqn:wave-equation-explicit-2}
    \begin{bmatrix}
        \sigma\\
        v
    \end{bmatrix}
    = \begin{bmatrix}
        E & 0 \\ 0 & \frac{1}{\rho}
    \end{bmatrix} \begin{bmatrix}
        \varepsilon \\ p
    \end{bmatrix}.
\end{equation}
In order to identify the boundary power port, one can compute the \emph{power balance} $\frac{\rm d}{{\rm d}t} H_W^{SD}$ along the trajectory of a solution, which yields:
\begin{equation} \label{eq:wave-1D-power-balance-explicit}
    \forall \, t>0, \quad\frac{\rm d}{{\rm d}t} H_W^{SD}(\varepsilon(t),p(t)) = [v \, \sigma]_a^b = \sigma(t,b) v(t,b) - \sigma(t,a) v(t,a) \, ,
\end{equation}
namely the product of the force and the velocity at both boundaries. 

\begin{remark}
    Denoting the flow variables as $f = (- \partial_t \varepsilon, - \partial_t p )^\top$ and   $f_\partial = (-v(t,a), v(t,b))^\top$; the effort variables: $e = (\sigma, v)^\top$ and $e_\partial = (\sigma(t,a), \sigma(t,b))^\top$; equation~\eqref{eq:wave-1D-power-balance-explicit} equivalently rewrites: $\langle e, f\rangle + e_\partial \cdot f_\partial = 0$. This shows the power preserving nature of this system, power is either put or removed at the boundary through the \emph{boundary power port} $(f_\partial,e_\partial)$; this motivates the definition of Stokes-Dirac structures, as will be enlightened in Section~\ref{sec:Stokes-Lagrange}. 
\end{remark}

A choice of state variables was made in \eqref{eqn:ham-wave-explicit}, leading to the Hamiltonian given as an explicit functional of the state variables. However, this choice is not canonical. Hence, two questions naturally arise: What could a different choice of variables (\textit{i.e.}, representation) yield? And what happens when the Hamiltonian is not an explicit functional of the state variables anymore? 

\subsubsection{A different representation}

Let us now choose a different set of variables, \textit{e.g.}, the displacement $w$ and the linear momentum $p$, then the Hamiltonian~\eqref{eqn:ham-wave} reads:
\begin{equation}\label{eqn:ham-wave-implicit}
    H_W^{SL}(w,p) = \frac{1}{2} \int_a^b \frac{1}{\rho}p^2 + E (\partial_x w)^2 \, \d x \, .
\end{equation}
Computing the co-state variables yields $N := \delta_w H_W^{SL} = - \partial_x \, (E \, \partial_x w)$, the stress resultant, and $v := \delta_p H_W^{SL} = \frac{1}{\rho} \, p$, the velocity (see~\ref{apx-subsec:wave-varder-lagrange} for a detailed proof), which gives us the following pH system:
\begin{equation} \label{eqn:wave-equation-implicit-1}
    \frac{\partial}{\partial t}\begin{bmatrix}
        w \\
        p
    \end{bmatrix} = \begin{bmatrix}
        0 & 1 \\ -1 & 0
    \end{bmatrix} \begin{bmatrix}
        N \\ v
    \end{bmatrix},
\end{equation}
together with the constitutive relations:
\begin{equation} \label{eqn:wave-equation-implicit-2}
    \begin{bmatrix}
        N \\ v
    \end{bmatrix}
    =
    \begin{bmatrix}
        - \partial_x ( E \partial_x \cdot ) & 0 \\
        0 & \frac{1}{\rho}
    \end{bmatrix}
    \begin{bmatrix}
        w \\ p
    \end{bmatrix}.
\end{equation}

\begin{remark}
    With this representation~\eqref{eqn:wave-equation-implicit-1}--\eqref{eqn:wave-equation-implicit-2}, the differential operator $\partial_x$ appears in the Hamiltonian~\eqref{eqn:ham-wave-implicit}, hence in the constitutive relations~\eqref{eqn:wave-equation-implicit-2}, and not in the structure matrix operator~\eqref{eqn:wave-equation-implicit-1}, as in~\eqref{eqn:wave-equation-explicit-1}. This will require a broader notion of a Hamiltonian and be further discussed in Section~\ref{sec:Stokes-Lagrange}.
\end{remark}
Finally, one can compute the power balance as follows:
\begin{equation} \label{eq:wave-1D-power-balance-implicit}
    \forall \, t>0, \quad\frac{\rm d}{{\rm d}t} H_W^{SL}(w,p) = [ (\partial_t w) \, E (\partial_x w)]_a^b \, .
\end{equation}

\begin{remark} \label{rmk:power-balance-wave-equation}
The physical interpretation of this power balance~\eqref{eq:wave-1D-power-balance-implicit} is identical to that of the first representation~\eqref{eq:wave-1D-power-balance-explicit}, however the mathematical objects used are different (namely, the time derivative of a state variable and the Neumann trace of a state variable, instead of Dirichlet and normal traces of co-state variables). These boundary ports do not fit in the Stokes-Dirac framework anymore, as they are not functions of the effort variables $(N, v)$ but of the state variable $w$. Hence, an extension of the pH framework has to be settled. This has been done for a class of differential operators on 1D domains in~\cite{maschke2023linear} by using Lagrangian subspaces. The present work proposes an extension to $N$-dimensional domains, and by analogy with Stokes-Dirac structures, it will lead to the definition of Stokes-Lagrange structures.
\end{remark}

\subsubsection{Implicit Hamiltonian}

Another instance where the Stokes-Lagrange structure will prove beneficial involves certain constitutive relations, particularly those that are non-local.
Let us consider the representation~\eqref{eqn:wave-equation-explicit-1}, and say that one is studying a nanorod onto which the stress is nonlocal, \textit{i.e.}, the stress not only depends on the local strain but also on other points of the domain. In Eringen's original paper~\cite{eringen1983differential}, nonlocal constitutive relations are firstly presented as a kernel operator:
$$
    \sigma(t,x) = \int_a^b \alpha(|x-x'|) E \, \varepsilon(t,x') \d x' \, ,
$$
with the kernel $\alpha(x) := \frac{1}{2\sqrt{\mu}} e^{-\frac{1}{\sqrt{\mu}} x}$ and $\mu>0$; $\sqrt{\mu}$ being the characteristic length of the nonlocal effects. Using a Fourier transform and assuming that $\sqrt{\mu} <\! \!< (b-a)$ one gets~\cite{eringen1983differential}:
$$
    (1- \mu \partial_{x^2}^2) \, \sigma = E \, \varepsilon \, .
$$
Note that these formulations are equivalent if one considers that both variables vanish at the boundary and outside $[a,b]$.
\paragraph{Modelling aspects}It turns out that on top of some functional analysis argument (allows for the study of a partial differential equation instead of an integro-partial differential equation), this formulation is much better numerically as it will be shown in the following. Such a constitutive relation is given in \textit{e.g.}~\cite{eringen1983differential,heidari2019port} and leads to the constitutive relations:
\begin{equation}\label{eqn:stress-nonlocal}
    \begin{bmatrix}
        (1 - \mu \, \partial_{x^2}^2) & 0 \\
        0 & I
    \end{bmatrix} \begin{bmatrix}
        \sigma \\ v
    \end{bmatrix} = \begin{bmatrix}
        E & 0 \\ 0 & \frac{1}{\rho}
    \end{bmatrix} \begin{bmatrix}
        \varepsilon \\ p
    \end{bmatrix},
\end{equation}
with $\mu > 0$ the nonlocality parameter.

This constitutive relation is implicit in the sense that an (unbounded) operator is to be found on the left-hand side of~\eqref{eqn:stress-nonlocal}, \textit{i.e.}, in front of the co-state variables. Hence, the Hamiltonian is not an explicit function of the state variables. Indeed, the co-state variables are, by definition, the variational derivatives of the Hamiltonian with respect to the state variables. Then, the presence of the operator on the left-hand side of~\eqref{eqn:stress-nonlocal} corresponds to, roughly speaking, $\delta_\varepsilon H = (1-\mu \partial_{x^2}^2)^{-1} E \varepsilon$ -- one should be cautious about boundary terms in the existence and definition of an inverse here --, which does determine the Hamiltonian $H$, but not in a straightforward manner, or in other words, implicitly.

To solve this issue, two possibilities have been explored: extend the flow and effort space in order to put this relation as a constraint in the Dirac structure~\cite{heidari2022nonlocal}, or use a Lagrange subspace to define the Hamiltonian~\cite{maschke2023linear}. In this article, we will focus on the latter, by first extending this setting to $N$-dimensional spatial domains (which will then be called a Stokes-Lagrange structure) and present some examples.

When using such an approach, it is useful to define the \emph{latent state space}. In the above linear 1D wave example, the latent deformation $\lambda$ is defined by $(1-\mu \partial_{x^2}^2) \lambda =  \varepsilon$.
\begin{remark}
    Note that as the parameter $\mu$ tends to zero, the relation \eqref{eqn:stress-nonlocal} into the classical local stress-strain relation~\eqref{eqn:wave-equation-explicit-2} (\textit{i.e.}, Hooke's law), this property is rather important when dealing with such systems, in particular for numerics where different model parametrization might require very small values of $\mu$. In particular, when using the representation proposed in~\cite{heidari2022nonlocal}, this property is lost: the relation diverges to infinity as this parameter tends to zero.
\end{remark}

\paragraph{Numerical aspects} Let us now discuss the advantages of such an implicit relation over an explicit nonlocal relation. Let $(\phi_i)_{i \in [1,N]}$ be a family of $\mathbb P^1$ Lagrange finite elements over $[a,b]$, then defining $\sigma^d(t,x) := \sum_{i} \phi_i(x) \overline{\sigma}_i(t)$ and $\varepsilon^d(t,x) := \sum_{i} \phi_i(x) \overline{\varepsilon}_i(t)$, one gets:
$$
    \forall \, i \in [1,N], \quad \sum_j \underbrace{\int_a^b \phi_i \phi_j}_{M_{i,j}} \, \overline{\sigma} = \sum_j \underbrace{\int_a^b\int_a^b \alpha(|x-x'|) \phi_i(x) E \phi_j(x') \d x \d x'}_{K^\alpha_{i,j}} \overline{\varepsilon} \, ,
$$
$$
    \forall \, i \in [1,N], \quad \sum_j \left[\int_a^b \phi_i \phi_j +  \underbrace{\int_a^b (\partial_x \phi_i \mu \partial_x \phi_j)}_{K^\mu_{i,j}}\right] \,\overline{\varepsilon} = \sum_j \underbrace{\int_a^b \phi_i E \phi_j}_{K^{E}_{i,j}} \, \overline{\varepsilon} \, .
$$
More compactly we get:
$$
\begin{aligned}
        M \, \overline{\sigma} =& K^\alpha \, \overline{\varepsilon} \, , & \quad \text{(Explicit formulation)}, \\
        (M +  K^\mu) \, \overline{\sigma} =& K^{E} \, \overline{\varepsilon} \, , &\quad \text{(Implicit formulation)}.
\end{aligned}
$$
Then, in both cases a linear system has to be solved (since $M \neq I_N$ and $ (M + K^\mu) \neq I_N$), however in the implicit case, the matrices are sparse whereas in the explicit case, a double integration has to be carried (which slows the assembly of the matrix) and $K^\alpha$ is \textbf{full}, which greatly degrades the numerical simulation CPU time and memory usage (see Figure~\ref{fig:impex-comparison}). In particular, the memory usage increases ($O(N^2)$ nonzero components instead of $O(N)$).
\begin{figure}[ht]
    \centering
    \includegraphics[width=.7\textwidth]{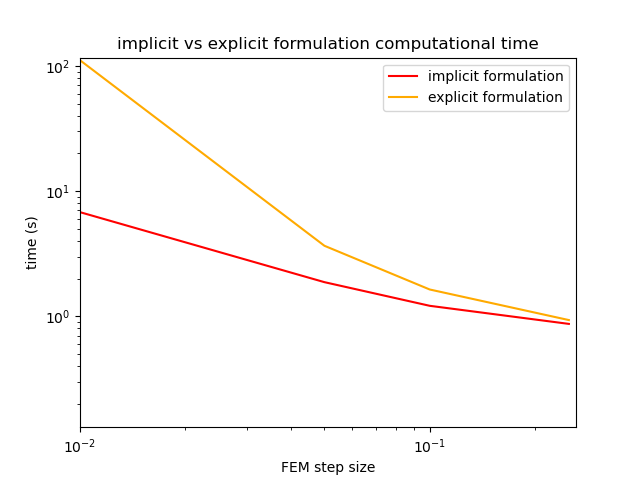}
    \caption{Comparison of the assembly + solving time of the implicit vs explicit formulations.}
    \label{fig:impex-comparison}
\end{figure}

\subsection{Operator transposition} \label{subsec:operator-transposition}

As presented in the previous subsection, the Stokes-Lagrange structure allows us to describe implicit constitutive relations. Let us now show that this Stokes-Lagrange structure gives us different choices of representations and that one can choose whether the differential operators are to be found in the Dirac structure or in the Hamiltonian~\cite{SchoberlSiukaECC2013}.  Throughout the article, given an operator $A$ we will denote by $A^*$ the adjoint of $A$~\cite{RudinWalter1921}.

We will consider the general case as follows (detailed assumptions and proofs are provided in~\ref{apx:operator-transposition}):
let us consider a spatial domain $\Omega \subset \mathbb{R}^N$, and the following wave-like system:
\begin{equation} \label{eqn:sys-K-matrix}
    \begin{cases}
        \partial_t \alpha_1 = \bm{\kappa} \alpha_2,\\
        \partial_t \alpha_2 = - \bm{K}^\dag(\bm{\eta} \bm{K} \alpha_1), 
    \end{cases}
\end{equation}
with $\alpha_1,\alpha_2$ of size $n$ and $\bm{\eta},\bm{\kappa} $ some bounded from above operators of appropriate sizes. $\bm{K}$ is a closed and densely-defined operator $\bm{K}: \mathcal{D}(\bm{K}) \subset L^2(\Omega, \mathbb{R}^n) \rightarrow L^2(\Omega, \mathbb{R}^m)$ satisfying a Stokes' identity. More precisely, for all $\phi \in D(\bm{K}), \, \psi \in D(\bm{K}^\dag)$:
$$
\int_\Omega \bm{K} (\phi) \cdot \psi \, \d x
=
\int_\Omega \phi \cdot \bm{K}^\dag (\psi) \, \d x
+ \left\langle  \beta(\psi), \gamma(\phi) \right\rangle_{\mathcal{U}, \mathcal{Y}},
$$
$\bm{K}^\dag$ being the \emph{formal adjoint}\footnote{$\bm{K}^\dag$ is the adjoint of $\bm{K}$ ``up to boundary terms" \textit{i.e.} it is the adjoint of $\bm{K}$ in the subspace $ \ker(\gamma) \subset D(\bm{K})$, more precisely $(\bm{K}_{|\ker(\gamma)})^* = \bm{K}^\dag$ and $(\bm{K}^\dag_{|\ker(\beta)})^* = \bm{K}$.}  of $\bm{K}$, $\mathcal{U}$ and $\mathcal{Y} := \mathcal{U}'$ the control and observation subspaces and $\gamma:D(\bm{K})\rightarrow\mathcal{Y}$ and $\beta:D(\bm{K}^\dag)\rightarrow\mathcal{U}$ the boundary operators.  
The associated Hamiltonian of \eqref{eqn:sys-K-matrix} is:
$$
    H := \frac{1}{2}\int_\Omega  (\bm{K} \alpha_1) \cdot \bm{\eta} (\bm{K} \alpha_1) + \alpha_2 \cdot \bm{\kappa} \, \alpha_2 \, \d x\geq 0 \, . 
$$
Now, in order to represent \eqref{eqn:sys-K-matrix} as a pH system, one can choose as energy variable either  $(\alpha_1^{SD},\alpha_2^{SD}) = (\bm{K}\alpha_1, \alpha_2)$ which removes the operator from the Hamiltonian, or one can choose $(\alpha_1^{SL},\alpha_2^{SL}) = (\alpha_1,\alpha_2)$. This yields the two following pH systems:
\begin{equation} \label{eqn:stokes-dirac-sys-k}
        \begin{bmatrix}
            \partial_t \alpha^{SD}_1 \\
            \partial_t \alpha^{SD}_2
        \end{bmatrix} = \underbrace{\begin{bmatrix}
            0 &  \bm{K}  \\ -\bm{K}^\dag & 0
        \end{bmatrix}}_{=: J^{SD}} \begin{bmatrix}
            e_1^{SD} \\ e_2^{SD}
        \end{bmatrix}, \quad
    \begin{cases}
        e^{SD}_1 =\bm{\eta}\, \alpha_1^{SD}, \\
        e^{SD}_2 =  \bm{\kappa} \, \alpha_2^{SD},
    \end{cases}
\end{equation}
being called the \emph{Stokes-Dirac representation}, and:
\begin{equation}\label{eqn:stokes-lagrange-sys-k}
        \begin{bmatrix}
            \partial_t \alpha^{SL}_1 \\
            \partial_t \alpha^{SL}_2
        \end{bmatrix} = \underbrace{\begin{bmatrix}
            0 & \Id \\ - \Id & 0
        \end{bmatrix}}_{=: J^{SL}} \begin{bmatrix}
            e_1^{SL} \\ e_2^{SL}
        \end{bmatrix}, \quad
        \begin{cases}
        e^{SL}_1 = \bm{K}^\dag( \bm{\eta}  \bm{K} \alpha_1^{SL}), \\
        e^{SL}_2 = \bm{\kappa} \, \alpha_2^{SL},   
        \end{cases}
\end{equation}
being called the \emph{Stokes-Lagrange representation}. 
\begin{remark}
    We call a representation of a pH system \emph{Stokes-Dirac representation} when the differential operators are only present in the structure matrix; and we call it \emph{Stokes-Lagrange representation} when the differential operators are only present in the Hamiltonian. Note that some systems (\textit{e.g.} the nanorod equation) possess differential operators both in the Hamiltonian and the structure matrix~\cite{maschke2023linear}. Precise definitions and details about these two structures are given in Section~\ref{sec:Stokes-Lagrange}.
\end{remark}
Finally, defining $G$ along with its formal adjoint $G^\dag$ as: 
$$
G = \begin{bmatrix}
    \bm{K} & 0 \\
    0 & \Id
\end{bmatrix}, \qquad
    G^\dag = \begin{bmatrix}
    \bm{K}^\dag & 0 \\
    0 & \Id
\end{bmatrix},
$$
we get the following theorem:
\begin{theorem} \label{thm:dirac-lagrange-equiv} The Stokes-Dirac \eqref{eqn:stokes-dirac-sys-k} and Stokes-Lagrange \eqref{eqn:stokes-lagrange-sys-k} representations of the system \eqref{eqn:sys-K-matrix} satisfy the following equalities:
    \begin{equation} \label{eqn:operator-transposition}
        \begin{cases}
            \alpha^{SD} = G \alpha^{SL}, \\
            G^\dag e^{SD} = e^{SL}, \\
            GJ^{SL}G^\dag = J^{SD}.
        \end{cases}
    \end{equation}
    They are then said to be equivalent.
\end{theorem}
\begin{proof}
    See~\ref{apx:operator-transposition-representation}.
\end{proof}
\begin{remark}
    Note that the position of $G$ and $G^\dag$ are not the same between line one and two of \eqref{eqn:operator-transposition}. From a geometrical viewpoint, this can be explained by the fact that state variables are usually described as vectors whereas effort variables are described as covectors; leading to contravariant and covariant transformations respectively.
\end{remark}
\begin{remark}
    Note that $\bm{K}$ may have a nonzero kernel. Then, when using the Stokes-Dirac representation, a linear subspace of the state space is removed from the representation, namely $\ker(\bm{K})$ and by the definition of the Hamiltonian, this subspace exactly corresponds to states that have zero energy: $\forall \alpha \in \ker({\bm{K}}), \quad  \bm{K} \alpha \cdot \bm{\eta} \bm{K} \alpha = 0$. This will be commented in Example~\ref{example:wave-operator-transposition} where the rigid body motion is not taken into account in the Stokes-Dirac representation.
\end{remark}

\begin{remark}
    This procedure has been presented here using an operator theory point of view which translates directly to vector calculus; it has also been carried from a geometric point of view using jet extension in~\cite{preuster2024jet,SchoberlSiukaECC2013}.
\end{remark}

\begin{example}(Wave equation - N-dimensional case) \label{example:wave-operator-transposition}

Let us consider the wave equation on $\Omega \subset  \mathbb{R}^n$:
$$
    \partial_t(\rho  \, \partial_{t} w) = \diver(E \, \grad(w)) \, ,
$$
with the Hamiltonian $H = \frac{1}{2} \int_\Omega \rho \, (\partial_t w)^2 + E \, |\grad(w)|^2 \, \d x$.
Defining $\varepsilon := \partial_x w$ and $p := \rho \, \partial_t w$, one deduces the co-energy variables $\sigma = E \, \varepsilon$ and  $v = \frac{1}{\rho} p$ and gets the following Stokes-Dirac formulation:
$$
    \partial_t \begin{bmatrix}
        \varepsilon \\
        p
    \end{bmatrix} = \begin{bmatrix}
        0 & \grad \\ \diver & 0
    \end{bmatrix} \begin{bmatrix}
        \sigma \\ v
    \end{bmatrix}.
$$
Using directly $w$ and $p$, one can then write the Stokes-Lagrange representation as:
$$
    \partial_t \begin{bmatrix}
        w \\ p
    \end{bmatrix} = \begin{bmatrix}
        0 & 1\\ -1 & 0
    \end{bmatrix} \begin{bmatrix}
        \delta_w H \\ \delta_p H
    \end{bmatrix}, \quad \text{with} \quad \begin{bmatrix}
        \delta_w H \\ \delta_p H
    \end{bmatrix} = \begin{bmatrix}
        - \diver(E \grad(w))  \\ \frac{1}{\rho}p
    \end{bmatrix} .
$$
In particular, let us denote the transposition operator:
$$
G = \begin{bmatrix}
    \grad & 0 \\ 0 & \Id
\end{bmatrix}, \qquad
G^\dag = \begin{bmatrix}
    - \diver & 0\\
    0 & \Id
\end{bmatrix}, 
$$
which yields:
$$ 
\begin{bmatrix}
    \varepsilon \\ p
\end{bmatrix} = G \begin{bmatrix}
    w \\ p 
\end{bmatrix}, \qquad  
G^\dag \begin{bmatrix}
    \sigma \\ v
\end{bmatrix} = \begin{bmatrix}
    \delta_w H \\ v
\end{bmatrix}.
$$
\end{example}
\begin{remark}
Note that the kernel of $G$ is precisely the constant displacements: given $c\in\mathbb{R}$, $G \begin{bmatrix}
    c & 0
\end{bmatrix}^\top = 0$; hence in the Stokes-Dirac representation of the wave equation, the rigid body motion is lost~\cite{SchoberlSchlacherMathMod2015}. However, this is not the case if one extends the original state variables $(w,p)$ with the first jet $\partial_xw = \varepsilon$ while keeping the displacement $w$ in the state variable (see in particular~\cite{preuster2024jet}).
\end{remark}

To conclude and motivate the following section, let us summarize the important properties highlighted in the linear wave equation example.

\subsection{Summary and main results}

The pH formulation of the wave equation exhibits the following properties: this system can be decomposed into two structures; one that defines the energy linked to the state variables, and one that describes how this energy is exchanged between the ports. The former is usually called the Hamiltonian and the latter the Dirac structure. 

Let us emphasize once again that the usual Hamiltonian definition is not sufficient to describe implicit relations and does not take into account boundary terms when differential operators are to be found in the Hamiltonian.

Hence, as in the case of the Dirac structure, we need to extend the energy definition as a Hamiltonian to a so-called Lagrange subspace to account for implicit relations and as in the  case of the Stokes-Dirac structure~\cite{kurula2010dirac}, we need to define a special case of the Lagrange subspace, namely a Stokes-Lagrange structure to include boundary terms~\cite{maschke2023linear} when differential operators are present.

\section{Stokes-Dirac and Stokes-Lagrange structures}
\label{sec:Stokes-Lagrange}

The Stokes-Dirac structure is already well known in the pH community, and we will simply recall the basic facts presented in~\cite{kurula2010dirac} that will motivate our choices regarding the Stokes-Lagrange structure. The Stokes-Lagrange structure has only been studied on 1D domains with a restricted set of operators~\cite{maschke2023linear}, here we will extend it to $N$-dimensional domains. To do so, let us firstly define the Dirac and Lagrange structures.

\subsection{Dirac and Lagrange structures}
Let us consider a Hilbert space $\mathcal{P}$, then define the effort space $\mathcal{X} = \mathcal{F} \subset \mathcal{P}$ as a dense Hilbert subspace and the flow and state space as $ \mathcal{E} = \mathcal{F}'$, the dual of $\mathcal{F}$ with respect to the pivot space $\mathcal{P}$.
\begin{remark}
    Identifying the state and flow spaces $\mathcal{X}$ and $\mathcal{F}$ is only possible in the linear case (where we currently restrain ourselves), in the nonlinear case, given a manifold $\mathcal{M}$, the Dirac structure is defined on $T\mathcal{M} \bigoplus T^* \mathcal{M}$ and the Lagrange submanifold on $T^* \mathcal{M}$ (see \emph{\cite{krhac2024port,van2018generalized,van2020dirac}} for details).
\end{remark}
Now following~\cite{maschke2023linear} let us define the \emph{plus pairing} and \emph{minus pairing} as:
$$
\begin{aligned} \forall \begin{bmatrix}
        f_1 \\ e_1
    \end{bmatrix}, \begin{bmatrix}
        f_2 \\ e_2
    \end{bmatrix} \in \mathcal{F} \times \mathcal{E}, \; 
    \left\langle  \left\langle \begin{bmatrix}
        f_1 \\ e_1
    \end{bmatrix}, \begin{bmatrix}
        f_2 \\ e_2
    \end{bmatrix} \right\rangle \right\rangle_+  = \langle e_2,f_1\rangle_{\mathcal{E},\mathcal{F}} + \langle e_1,f_2\rangle_{\mathcal{E},\mathcal{F}} \, , \\
    \forall \begin{bmatrix}
        x_1 \\ e_1
    \end{bmatrix}, \begin{bmatrix}
        x_2 \\ e_2
    \end{bmatrix} \in \mathcal{X} \times \mathcal{E}, \; 
    \left\langle  \left\langle \begin{bmatrix}
        x_1 \\ e_1
    \end{bmatrix}, \begin{bmatrix}
        x_2 \\ e_2
    \end{bmatrix} \right\rangle \right\rangle_-  = \langle e_2,x_1\rangle_{\mathcal{E},\mathcal{X}} - \langle e_1,x_2\rangle_{\mathcal{E},\mathcal{X}} \, .
\end{aligned}
$$
Then, defining the orthogonal of a subspace with respect to the plus pairing as $\perp_+$, \textit{i.e.}: 
$$
    \forall \mathcal{Q}\subset \mathcal{F}\times\mathcal{E}, \quad \mathcal{Q}^{\perp_+} := \left\lbrace z \in \mathcal{F}\times\mathcal{E} \, \mid \, \forall \tilde z \in \mathcal{Q}, \langle\langle z, \tilde z \rangle\rangle_+ = 0 \right \rbrace \, ,
$$
and similarly with respect to the minus pairing as $\perp_-$, we can properly define Lagrange and Dirac structures as:
\begin{definition}
    A subspace $\mathcal{D} \subset \mathcal{F}\times\mathcal{E}$ is called a Dirac structure if:
    $$\mathcal{D} = \mathcal{D}^{\perp_+}.$$
\end{definition}
In particular a Dirac structure is isotropic $\mathcal{D} \subset \mathcal{D}^{\perp_+}$ and co-isotropic $\mathcal{D}^{\perp_+} \subset\mathcal{D}$ with respect to the plus pairing.
\begin{definition}
    A subspace $\mathcal{L} \subset \mathcal{X}\times\mathcal{E}$ is called a Lagrangian subspace if:
    $$\mathcal{L} = \mathcal{L}^{\perp_-}.$$
\end{definition}
In particular, a Lagrangian subspace is isotropic $\mathcal{L} \subset \mathcal{L}^{\perp_-}$ and co-isotropic $\mathcal{L}^{\perp_-} \subset\mathcal{L}$ with respect to the minus pairing.

Let us now illustrate the meaning of the two previous definitions.
The Dirac structure allows for describing the energy routing in a system (or equivalently entropy/free energy/free enthalpy...), and to represent a system as a network of subsystems exchanging energy through power ports. These ports are usually of three different natures (see~\cite{van2014port}): energy storage ports, input/output (control) ports and resistive ports.  
Given a set of storage ports $(f_s,e_s)$, resistive $(f_R,e_R)$ and control ports $(f_u,e_u)$ that belongs to a certain Dirac structure $\mathcal{D}$, one can deduce from the \emph{isotropy} the following \textbf{power balance}:
$$ 
\langle e_s,f_s \rangle_{\mathcal{E}_s,\mathcal{F}_s}  + \langle e_R,f_R \rangle_{\mathcal{E}_R,\mathcal{F}_R} + \langle e_u,f_u \rangle_{\mathcal{E}_u,\mathcal{F}_u} = 0 \, , 
$$
which states that given a port $(f_i,e_i)$, the power provided/received by this  port expressed as $\langle e_i,f_i  \rangle$ (\textit{e.g.} force$\times$velocity, current$\times$voltage) is preserved and can only be provided (resp. received) to (resp. by) other subsystems. 

The Lagrangian subspace defines the energy of the system, by extending the classical Hamiltonian definition~\cite{krhac2024port,van2018generalized,van2020dirac}.
Isotropy and co-isotropy with respect to the minus product enforces the Maxwell's reciprocity conditions (see Definition~\ref{def:maxwell-reciprocity-condition}) on the constitutive relations~\cite{gangi2000constitutive}. This guarantees the existence of a controlled Hamiltonian $H(z, \chi)$, with $(\chi, \delta_\chi H)$ being the energy control port (see~\ref{apx:maxwell-reciprocity} for a detailed exposition) and $z$ the latent state variable.
\begin{remark}
    Given an implicit linear relation between the state and effort variable as $ S^\top \alpha = P^\top e$ with $\alpha,e\in \mathbb{R}^n$ and $S,P \in \mathbb{R}^{n\times n}$, the Maxwell's reciprocity conditions read:
    $$ 
    S^\top P = P^\top S \, , 
    $$
    they ensure that a corresponding latent state variable $z\in \mathbb{R}^n$ such that $\alpha = Pz$, $e = Sz$ exists, along with a Hamiltonian $H(z)$ that satisfies the power balance:
    $$
    \frac{\rm d}{{\rm d}t}H(z) = \langle \partial_t \alpha, e\rangle \, .
    $$
     In the linear case, these conditions can be written as~\cite{gernandt2022equivalence}:
    $$  
    S^\top P - P^\top S = 0 \, ,
    $$
    which, given $i \in \lbrace 1,2\rbrace$, $\begin{bmatrix}
        \alpha_i \\ e_i
    \end{bmatrix} = \begin{bmatrix}
        P \\ S
    \end{bmatrix} z_i$ is equivalent to:
    $$ 
    \left\langle \left\langle \begin{bmatrix}
        \alpha_1\\ e_1
    \end{bmatrix}, \begin{bmatrix}
        \alpha_2 \\ e_2
    \end{bmatrix} \right\rangle \right\rangle_- = 0 \, .
    $$
    Note that in the nonlinear case, the Lagrange subspace becomes a Lagrangian submanifold~\cite{van2020dirac}. See~\ref{apx:maxwell-reciprocity} for a proof of the existence of a Hamiltonian given the Maxwell's reciprocity conditions.
\end{remark}

In this work, the Lagrangian subspace is referred to as a Lagrange structure for its similarities with the Dirac structure (isotropy and co-isotropy with respect to the plus/minus pairing~\cite{mehrmann2023differential}). These two structures will give rise to equivalent representations as will be discussed later; note that such matter has been studied in the finite-dimensional case in~\cite{gernandt2022equivalence}.

\begin{remark}
    The Dirac and Lagrange structures can be studied in the broader notion of Kreĭn spaces~\cite{bognar2012indefinite} using the indefinite plus and minus product respectively; they would then correspond to \emph{hyper-maximal neutral} subspaces of Kreĭn spaces equipped with the plus and minus product respectively. Such an approach has been suggested in~\cite{kurula2010dirac} for Dirac structures.
\end{remark}

\subsection{Defining a port-Hamiltonian system}
A pH system is defined by a Dirac structure $\mathcal{D}$, a Lagrange structure $\mathcal{L}$ and a resistive subspace $\mathcal{R}$ such that\footnote{Note that $f_s = - \partial_t \alpha$ is due to the sign convention and allows us to write $\langle e, f\rangle = 0$.}:
$$
\begin{aligned}
    \underbrace{\begin{bmatrix}
        f_s & f_R & f_u & e_s & e_R & e_u
    \end{bmatrix}^\top \in \mathcal{D}}_{\text{Energy routing}} \, , & \qquad \underbrace{\begin{bmatrix}
        \alpha & \chi & e_s  & \varepsilon
    \end{bmatrix}^\top \in \mathcal{L}}_{\text{Energy definition}} \, , \\ \\
    \underbrace{f_s = - \partial_t \alpha}_{\text{Dynamics}} \, , & \qquad  \underbrace{(f_R,e_R) \in \mathcal{R}}_{\text{Resistive relation}} \, .
\end{aligned}
$$
$(f_s,e_s), (f_u,e_u)$ and $(f_R,e_R)$ are the state, control and resistive power ports respectively;  $(\alpha,e_s)$ and $(\chi,\varepsilon)$ are the state and control energy ports.
In particular, the resistive relation $(f_R,e_R) \in \mathcal{R}$ gives us $\langle e_R, f_R\rangle \leq 0$ (by definition of a resistive structure~\cite{van2014port}). This pH system can be represented as the following diagram:
\begin{figure}[ht]
    \centering
    \begin{tikzpicture}
        \node (X) at (0,2) {$\alpha \in \mathcal{X}$};
        \node (F) at (-2,0) {$f \in \mathcal{F}$};
        \node (E) at (2,0) {$e \in \mathcal{E}$};
        \node (D) at (0,-.5) {$\mathcal{D}, \mathcal{R}$};
        \node (L) at (1.75,1.25) {$\mathcal{L}$};
        \node (T) at (-2,1.25) {$- \frac{\partial}{\partial t}$};
        \draw (X)--(F)--(E)--(X);
    \end{tikzpicture}
    \caption{Diagram of an implicitly defined port-Hamiltonian system.}
    \label{fig:schema-port-hamiltonian}
\end{figure}
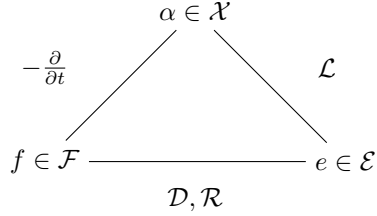

From the Lagrange structure, we get the existence of a Hamiltonian $H(z,\chi),$ $z$ being the latent state variable and $\chi$ the energy control variable. This Hamiltonian satisfies along the trajectories:
$$ 
\frac{\rm d}{{\rm d}t}H(z,\chi) = \langle e_s, \partial_t \alpha  \rangle_{\mathcal{E}_s, \mathcal{F}_s} + \langle \varepsilon, \partial_t \chi \rangle_{\mathcal{U},\mathcal{Y}} \, .
$$
Then, one can deduce from the Dirac and resistive structure the following power balance:
$$
\begin{aligned}
\frac{\rm d}{{\rm d}t} H(z, \chi ) = & \langle e_R,f_R \rangle_{\mathcal{E}_R,\mathcal{F}_R} + \langle e_u,f_u \rangle_{\mathcal{E}_u,\mathcal{F}_u} + \langle  \varepsilon,\partial_t \chi \rangle_{\mathcal{U},\mathcal{Y}} \\ \leq &  \langle e_u,f_u\rangle_{\mathcal{E}_u,\mathcal{F}_u} + \langle \varepsilon ,\partial_t  \chi\rangle_{\mathcal{U},\mathcal{Y}} \, .
\end{aligned}
$$
As a consequence, every pH system is cyclo-passive~\cite{ortega2008control} and is passive if $H\geq0$ (cyclo-passivity does not require the storage function $H$ to be positive whereas passivity does, see \textit{e.g.}~\cite{ortega2008control} for details).

\begin{remark}
    Even though pH systems are described here in implicit form (tuples of variable belonging to sets); explicit formulations of pH systems and in particular explicit representations of the Stokes-Dirac and Stokes-Lagrange structures are given in Subsection~\ref{subsec:an-explicit-representation}.
\end{remark}
\begin{remark}
    Note that  the Dirac structure and Lagrange structure are usually defined on different spaces: $(\mathcal{F}\times \mathcal{F}_{\mathcal{R}} \times \mathcal{F}_u) \times (\mathcal{E} \times \mathcal{E}_{\mathcal{R}} \times \mathcal{E}_u)$ and $(\mathcal{X} \times \mathcal{Y} \times  \mathcal{E} \times \mathcal{U})$ respectively, \textit{i.e.} the Dirac structure may contain a resistive structure $(\mathcal{F}_{\mathcal{R}},\mathcal{E}_{\mathcal{R}})$ and an input space with collocated output $(\mathcal{F}_u \times \mathcal{E}_u)$; whereas the Lagrange structure contains the state and costate $(\mathcal{X},\mathcal{E})$ along with an energy control port $(\mathcal{Y},\mathcal{U})$.
\end{remark}
\subsection{The Stokes-Dirac structure}

A Stokes-Dirac structure is a particular case of Dirac structures. More precisely, it allows us to explicitly take into account the boundary terms when differential operators define the dynamics. We will follow the definition presented in~\cite[Eq. (4.4)]{kurula2010dirac} where it is defined using operator colligation. To do so,
let us consider two Hilbert spaces  $\mathcal{F} = \mathcal{F}_s \times \mathcal{F}_u$ and $\mathcal{E} = \mathcal{E}_s \times \mathcal{E}_u$, the flow and effort spaces; $\mathcal{E}$ being the dual of $\mathcal{F}$.

Now, let $J: \mathcal{D}(J) \subset \mathcal{E}_s \rightarrow \mathcal{E}_s$ be the structure operator. Defining $\mathcal{W}_1 := \mathcal{D}(J)$, let $G: \mathcal{W}_1 \rightarrow \mathcal{E}_u$ and $K: \mathcal{W}_1 \rightarrow \mathcal{F}_u$ be the two observation and control operators; and $\mathcal{W}_0 = \ker(K) \cap \ker(G)$. Then, let us define the bond space $\mathcal{B} := \mathcal{F} \times \mathcal{E} $ and assume the following properties:
\begin{assumption}\label{eqn:skew-symmetry-assumption}(skew-symmetry, maximality and density)
    \begin{itemize}
        \item $\forall z \in \mathcal{W}_1, \quad \langle Jz, z\rangle_{\mathcal{E}_s} =  \langle G z, K z\rangle_{\mathcal{E}_u, \mathcal{F}_u} \, .$
        \item $J_0 := J_{|\mathcal{W}_0}$ is densely defined and satisfies $J_0^* = - J$.
        \item The range of $\begin{bmatrix}
            K \\ G
        \end{bmatrix}$ is dense in $\mathcal{F}_u \times \mathcal{E}_u$.
    \end{itemize}
\end{assumption}
The Stokes-Dirac structure can be defined as:
$$
    \mathcal{D} := \begin{bmatrix}
        J \\ K \\ \Id \\ G
    \end{bmatrix} \mathcal{W}_1 \subset \mathcal{B} \, .
$$
And Assumption~\ref{eqn:skew-symmetry-assumption} guarantees that it is a Dirac structure, \textit{i.e.} $\mathcal{D}^{\perp_+} = \mathcal{D}$ (see~\cite{kurula2010dirac}). 
\begin{remark}
Let us comment Assumption~\ref{eqn:skew-symmetry-assumption}; the first line is what is usually refereed to as a ``formally skew-symmetric'' operator, \textit{i.e.} skew-symmetric up to boundary terms; the second and third lines give us the co-isotropy of the Dirac structure and in particular the third lines ensures that the observation and control spaces are of the appropriate size. Assumption~\ref{eqn:skew-symmetry-assumption} will be used in~\ref{apx:operator-transposition} to prove that the wave-like system~\eqref{eqn:sys-K-matrix} is defined on a Stokes-Dirac structure.
\end{remark}
Note that this definition of a Dirac structure is explicit, in the sense that given $(f_s,f_u,e_s,e_u) \in \mathcal{D}$, we get by definition $f_s = J \, e_s$. In the next subsection, the Stokes-Lagrange subspace will allow for implicit relations, \textit{e.g.} having a (possibly unbounded) operator acting in front of $f_s$.

\subsection{The Stokes-Lagrange structure}

The goal of this section is to properly define a Stokes-Lagrange structure, and to give an explicit construction, as in~\cite{maschke2023linear}, by extending it to spatial domains with dimension greater than one. This is necessary as the following plate and electromagnetic examples require 2D and 3D domains respectively.

This will be carried out from a functional analysis point of view and not a differential geometric point of view as in~\cite{krhac2024port}, this allows us to look precisely at the effect of differential operators on the Lagrange structure, however only the linear case is explored, \textit{i.e.} systems with a quadratic Hamiltonian. 

Similarly to the Stokes-Dirac structure~\cite{brugnoli2023stokes,kurula2010dirac}, the goal here is to use differential operators in the Lagrange structure (\textit{i.e.} in the Hamiltonian) and to describe the role of the integration by parts in such a definition.

Firstly, let us consider a Hilbert space $\mathcal{Z}$ where the latent space will be defined, then let us define the control space as a Hilbert space $\mathcal{U}$ and the observation space $\mathcal{Y} := \mathcal{U}'$ as the dual of the control space. Additionally, let us define a Hilbert space $\mathcal{X}$ as the state space and $\mathcal{E} := \mathcal{X}'$, the effort space,  defined as the dual of the state space. We will now consider four operators, the first two $S$ and $P$ will allow us to define the Hamiltonian density in the spatial domain and the last two $\gamma$ and $\beta$ will allow us to take into account the integration by parts when $S$ or $P$ is a differential operator.

Let: $$P: \mathcal{D}(P) \subset \mathcal{Z} \rightarrow \mathcal{X}, \quad S: \mathcal{D}(S) \subset \mathcal{Z} \rightarrow \mathcal{E} \, ,$$ be two operators, in this work we will assume that either $P$ or $S$ is bounded on $\mathcal{Z}$ and that the other is densely-defined and closed, and define the latent space $\mathcal{Z}_1 = \mathcal{D} (P)$ or $\mathcal{Z}_1 = \mathcal{D} (S)$ accordingly.

\begin{assumption} \label{eqn:bounded-assumption} (Unbounded operator)
Either $P$ or $S$ is bounded on $\mathcal{Z}$.
\end{assumption}

\begin{remark}
    This simplification avoids the necessity of considering the intersection of operator domains $\mathcal{D}(P) \cap  \mathcal{D}(S)$ at the cost of only encompassing cases where at most one operator among $P$ and $S$ is unbounded. However, it does not imply that either $P$ or $S$ are invertible, hence constraints on the Lagrange structure are still allowed~\cite{van2018generalized}. In particular, the case where both $P$ and $S$ are bounded and $P$ is allowed to be singular is discussed in~\cite{erbay2024weierstra}.
\end{remark}
Now let us define the corresponding boundary observation and control operators:
$\gamma: \mathcal{Z}_1 \rightarrow \mathcal{Y}, \quad  \beta: \mathcal{Z}_1 \rightarrow \mathcal{U}$.
With those defined, let us denote by $\mathcal{Z}_0$ the subspace $\mathcal{Z}_0 := \ker(\gamma) \cap \ker(\beta)$ onto which controls and observations are set to zero.
Finally, let us state three assumptions on these operators that will allow us to properly define a Stokes-Lagrange structure:
\begin{assumption} \label{eqn:symmetry-assumption} (Symmetry) 
For all $z_1,z_2 \in \mathcal{Z}_1$ we have:
    $$
    \begin{aligned}
    &\langle \beta z_1, \gamma z_2\rangle_{\mathcal{U},\mathcal{Y}} + \langle S z_1, P z_2 \rangle_{\mathcal{E},\mathcal{X}} = \langle S z_2, P z_1 \rangle_{\mathcal{E},\mathcal{X}} + \langle \beta z_2, \gamma z_1\rangle_{\mathcal{U},\mathcal{Y}} \, .
    \end{aligned}        
    $$
\end{assumption}
This symmetry property is to be understood as an integration by parts \emph{on each operator}, where the operator $S^*P$ is \emph{formally} symmetric, \textit{i.e.} $S^*P = P^*S$ on $\mathcal{Z}_0$, or $S^*P$ is symmetric ``up to boundary terms''.

\begin{remark}
    Note that two boundary terms are present in Assumption~\ref{eqn:symmetry-assumption} whereas only one is required in Assumption~\ref{eqn:skew-symmetry-assumption}. This is due to the fact that two integrations by parts are required to pass from $P^*S$ to $S^*P$. For instance, with $P = \partial_{x^2}^2$, $S = \Id$, we have:
    $$
    \begin{aligned}
        \int_a^b \partial_{x^2}^2 \alpha_1 \, \alpha_2 \,  \d x & = - \int_a^b \partial_x \alpha_1 \, \partial_x \alpha_2  \,  \d x + [\partial_x \alpha_1 \, \alpha_2]_a^b \, , \\
        & = \int_a^b \alpha_1 \, \partial_{x^2}^2 \alpha_2 \,  \d x + [\partial_x \alpha_1 \, \alpha_2]_a^b -   [\alpha_1 \, \partial_x \alpha_2]_a^b \, .
    \end{aligned}
    $$
\end{remark}

\begin{assumption} \label{eqn:maximality-assumption} (Maximality) 
Let $x,e \in \mathcal{X} \times \mathcal{E}.$ If:
$$
    \forall z \in \mathcal{Z}_0, \quad \langle S z, x \rangle_{\mathcal{E},\mathcal{X}} - \langle e ,P z \rangle_{\mathcal{E},\mathcal{X}} = 0 \, ,
$$
then there exists $\tilde z \in \mathcal{Z}_1$ so that, $x = P\tilde z$, $e = S\tilde z$. Defining $S_0 := S_{|\mathcal{Z}_0}$ and $P_0 := P_{|\mathcal{Z}_0}$, we get more compactly (see~\ref{apx:stokes-lagrange} for a proof):
$$ 
\ker \begin{bmatrix}
    S_0^* & - P_0^*
\end{bmatrix} \subset \Ima \begin{bmatrix}
    P \\ S
\end{bmatrix}.
$$
\end{assumption}
\begin{definition}
    A couple $(P,S)$ defined on a Hilbert space $\mathcal{Z}_1$ is said to be \emph{maximally reciprocal}~\cite{maschke2023linear} with respect to a subspace $\mathcal{Z}_0 \subset \mathcal{Z}$ if it satisfies Assumption~\ref{eqn:maximality-assumption} for that subspace.
\end{definition}
\begin{remark}
This maximality assumption allows for the co-isotropy of the Lagrange structure to hold. It extends the second line of Assumption~\ref{eqn:skew-symmetry-assumption} to implicit relations.
\end{remark}
\begin{assumption} \label{eqn:surjectivity-assumption} (Density) The range of the operator
    $\begin{bmatrix}
        \gamma \\ \beta
    \end{bmatrix}$ is dense in $\mathcal{Y}\times\mathcal{U}$.
\end{assumption}
\begin{remark}
    The maximality assumption along with the symmetry allows us to formally state the Maxwell's reciprocity conditions. This reciprocity condition is sufficient in the finite-dimensional case  see~\cite{hairer2006geometric}; However, the density assumption is mandatory in the infinite-dimensional case where we are now considering boundary terms appearing because of integration by parts.
\end{remark}
This allows us to define properly a Stokes-Lagrange structure:
\begin{definition}\label{def:Stokes-Lagrange-candidate} Given four operators $P,\,S,\,\gamma,\,\beta$ satisfying assumptions~\ref{eqn:bounded-assumption} ,~\ref{eqn:symmetry-assumption},~\ref{eqn:maximality-assumption} and~\ref{eqn:surjectivity-assumption}, a Stokes-Lagrange structure is defined as: 
$$
    \mathcal{L} := \begin{bmatrix}
        P \\ \gamma \\ S \\ \beta
    \end{bmatrix} \mathcal{Z}_1 \, ,
$$ 
such a definition is called the \emph{image representation} of the Stokes-Lagrange structure.
\end{definition}
Let us extend the minus product on the augmented space $\mathcal{X} \times \mathcal{Y} \times \mathcal{E} \times \mathcal{U}$:
$$
    \left\langle\left\langle \begin{bmatrix}
        x_1 \\ \chi_1 \\ e_1 \\ \varepsilon_1
    \end{bmatrix}, \begin{bmatrix}
        x_2 \\ \chi_2 \\ e_2 \\ \varepsilon_2
    \end{bmatrix} \right\rangle \right\rangle_-= \langle e_2, x_1 \rangle_{\mathcal{E},\mathcal{X}} - \langle e_1, x_2 \rangle_{\mathcal{E},\mathcal{X}} + \langle \varepsilon_2, \chi_1 \rangle_{\mathcal{U},\mathcal{Y}} - \langle \varepsilon_1, \chi_2 \rangle_{\mathcal{U},\mathcal{Y}} \, .
$$
Denoting with a $\perp \! \!\!\!\perp_-$ the orthogonal of a subspace in  $\mathcal{X} \times \mathcal{Y} \times \mathcal{E} \times \mathcal{U}$  with this augmented minus product, we are now able to state the following theorem:
\begin{theorem} \label{thm:lagrange-iso-coisotropic} (Isotropy and co-isotropy of the Stokes-Lagrange structure) Given $\mathcal{L}$ as in Definition~\ref{def:Stokes-Lagrange-candidate} we have:
    $$
        \mathcal{L} = \mathcal{L}^{ \perpext},
    $$
\textit{i.e.} a Stokes-Lagrange structure is a Lagrange structure.
\end{theorem}
\begin{proof}
    See~\ref{apx:stokes-lagrange-proof}.
\end{proof}
\begin{remark}
    This construction is very similar to the (Stokes-)Dirac structure case presented earlier, where $S = L$, $\gamma = K$, $\beta = G$, and $P = \Id$ and the plus product is used instead of the minus product. As mentioned in~\cite{kurula2010dirac}, these subspaces can be understood in the broader notion of Kreĭn spaces~\cite{bognar2012indefinite}.
\end{remark}
Definition~\ref{def:Stokes-Lagrange-candidate} and Theorem~\ref{thm:lagrange-iso-coisotropic} give us a framework onto which we shall build many examples later. Let us come back to the wave equation example to see how a Hamiltonian can be obtained from the Stokes-Lagrange structure.
\begin{example}(Wave equation - continued)
    Defining the operator $P = I_2$ and $S = \textup{Diag}\begin{bmatrix}
        -\partial_x E \partial_x & \frac{1}{\rho}
    \end{bmatrix}$, one gets:
    $$
        \begin{cases}
            P \partial_t \begin{bmatrix}
                w \\ p
            \end{bmatrix} = \begin{bmatrix}
                0 & 1 \\ -1 & 0
            \end{bmatrix} e \, , \\
            e = S x \, .
        \end{cases}
    $$
In particular, since $P = \Id_2$, the latent state $z$ is equal to the state $x$. One can define a Hamiltonian as the total mechanical energy and obtain: 
$$
\begin{aligned}
    H(w,p) &:= \frac{1}{2}  \int_a^b \left ( \frac{1}{\rho}p^2 + (E \partial_x w) (\partial_x w) \right ) \, \d x \, , \\
        &= \frac{1}{2} \int_a^b  \left ( \frac{1}{\rho} p^2 - \partial_x (E \partial_x w) w \right )  \, \d x + \frac{1}{2}[w E \partial_x w]_a^b \, , \\
        & = \frac{1}{2} \langle P \alpha, S\alpha\rangle_{L^2([a,b])} + \frac{1}{2} \langle \beta(\alpha), \gamma(\alpha)\rangle_{\mathbb{R}^2} \, , 
\end{aligned}
$$
where $\alpha := \begin{bmatrix} w \\ p \end{bmatrix}$ and the boundary operators are defined by $\gamma (\alpha) := \begin{bmatrix}
    w(a) \\ w(b)
\end{bmatrix}$ (the Dirichlet trace) and $\beta(\alpha) := \begin{bmatrix}
    - E \partial_x w(a) \\
    E \partial_x w(b)
\end{bmatrix}$ (the Neumann trace).
\end{example}
Following the previous example, let us now define the Hamiltonian of a general Stokes-Lagrange structure as:
\begin{equation}\label{eqn:Hamiltonian-Lagrange-Definition}
    \forall \, z \in \mathcal{Z}_1, \quad H(z) := \frac{1}{2} \langle S z, P z \rangle_{\mathcal{E},\mathcal{X}} + \frac{1}{2} \langle \beta(z), \gamma(z) \rangle_{\mathcal{U},\mathcal{Y}} \, .
\end{equation}
\begin{remark}
    Note that this Hamiltonian is explicit with respect to the latent state variable $z$ but not with respect to the state variable $x$, it is the case only if $P = \Id$, \textit{i.e.} $x = Pz = z$, as in the latter example.
\end{remark}
\subsection{An explicit representation} \label{subsec:an-explicit-representation}
Now, considering a port-Hamiltonian system implicitly defined with a Stokes-Dirac structure $\mathcal{D}$ and a Stokes-Lagrange structure $ \mathcal{L}$ as:
$$
\begin{bmatrix}
    f & f_u & f_R & e & e_u & e_R
\end{bmatrix} \in \mathcal{D} \, , \quad \begin{bmatrix}
    x & \chi & e & \varepsilon
\end{bmatrix} \in \mathcal{L} \, , \quad f = - \partial_t x \, ,
$$
with $(f,e)$ the state flow and effort power port, $(f_R,e_R)$ the Resistive power port ($\langle f_R, e_R\rangle \leq 0$), $(f_u,e_u)$ the control power port and $(\chi,\varepsilon)$ the control \emph{energy} port. Denoting $z$ the latent space variable, $S$ and $P$ the Lagrange structure operators, $J$ the formally skew-symmetric Dirac structure operator and $R = R^* \geq 0$ the resistive structure operator, we get the following explicit dynamics:
$$
\left \lbrace
    \begin{aligned}
    \begin{bmatrix}
     P \partial_t z \\
     - f_R 
    \end{bmatrix} &= J \begin{bmatrix}
        e \\ e_R
    \end{bmatrix},\\
    e &= S z \, , \\
    f_R &= - R \, e_R \, .
    \end{aligned} \right.
$$
Along the trajectories, this system satisfies the following power balance:
$$
\begin{aligned}
    \frac{\rm d}{{\rm d}t} H(z) & = \langle S z, P \partial_t z \rangle_{\mathcal{E}_s,\mathcal{F}_s}  + \frac{1}{2}(\langle \beta(z), \gamma (\partial_t  z)\rangle_{\mathcal{U},\mathcal{Y}} - \langle \beta(\partial_t z), \gamma (z) \rangle_{\mathcal{U},\mathcal{Y}}) \\ & \quad + \frac{1}{2}(\langle \beta(z), \gamma(\partial_t  z)\rangle_{\mathcal{U},\mathcal{Y}}  + \langle \beta (\partial_t z), \gamma(z) \rangle_{\mathcal{U},\mathcal{Y}}) \, , \\
    & = \langle e, \partial_t x \rangle_{\mathcal{E}_s,\mathcal{F}_s} + \langle \varepsilon, \partial_t \chi \rangle_{\mathcal{U},\mathcal{Y}} \, , \\
    & = \langle e_u, f_u \rangle_{\mathcal{E}_u,\mathcal{F}_u}  + \underbrace{\langle e_R, f_R \rangle_{\mathcal{E}_R,\mathcal{F}_R}}_{\leq0} +  \langle \varepsilon, \partial_t \chi \rangle_{\mathcal{U},\mathcal{Y}} \, , \\
    & \leq \langle e_u, f_u \rangle_{\mathcal{E}_u,\mathcal{F}_u}  +\langle \varepsilon, \partial_t \chi \rangle_{\mathcal{U},\mathcal{Y}} \, .
\end{aligned}
$$
Note that an additional port appears in the power balance, namely the \emph{energy port} $(\chi, \varepsilon)$. Moreover, the definition of a Hamiltonian is not unique~\cite{maschke2023linear}. Indeed, $H$ depends on $z$ and $\chi$ and one can perform a \emph{Legendre transform} with respect to $\chi$ and get the following:
$$
    \forall \, z \in \mathcal{Z}_1, \quad \tilde H(z) = \frac{1}{2} \langle S z, P z \rangle_{\mathcal{E},\mathcal{X}} - \frac{1}{2} \langle \beta z, \gamma z \rangle_{\mathcal{U},\mathcal{Y}} \, ,
$$
which yields:
$$
    \frac{\rm d}{{\rm d}t} \tilde H(z) = \langle e_u, f_u \rangle_{\mathcal{E}_u,\mathcal{F}_u} + \langle e_R, f_R \rangle_{\mathcal{E}_R,\mathcal{F}_R} - \langle  \partial_t \varepsilon,\chi \rangle_{\mathcal{U},\mathcal{Y}} \, .
$$
From a control perspective, choosing $\tilde H$ or $H$ gives us access to different input and output variables (the derivative of the input or the integral of the output) which will be discussed in detail in Section~\ref{sec:control}.

\begin{remark}
    Note that this representation of a Dirac and Lagrange structures is not the only possibility; in particular one could choose to put a partition of the effort variable on the left side of the first line and the corresponding flows on the right side yielding \textit{e.g.} $\begin{bmatrix}
        f_1 & e_2
    \end{bmatrix}^\top = J \begin{bmatrix}
        e_1 & f_2
    \end{bmatrix}^\top$, such representation has been studied in \cite{altmann2024novel}. In particular, since $J = -J^\top$, these systems are also defined on a Dirac structure with an explicit Hamiltonian; hence they are still pH systems.
\end{remark}

\section{Examples} \label{sec:examples}
Let us now apply the previous method to a set of examples, namely two plate models, the Maxwell's equation and the Dzektser seepage model. In particular, we will focus on the simplification from the Reissner-Mindlin thick plate model to the Kirchhoff-Love thin plate model~\cite{destuynder2013mathematical}  and from the classical Maxwell's equation to its low-frequency approximation. This will show us the different boundary ports and power balance obtained depending on the choice of representation, as well as how the constrained systems are written in either Stokes-Dirac or Stokes-Lagrange representations.

\subsection{Plate models}

Let us define the two following operators:
$$
    \begin{aligned}
        \forall \, \bm v \in H^1(\Omega, \mathbb R ^2),\qquad \Grad(\bm v) &:= \frac{1}{2}(\nabla \bm v + \nabla \bm v ^\top) \, , \\
        \forall \, \tensor T \in H^1(\Omega, \mathbb R ^{2\times 2}),\qquad \Diver(\tensor T)_{i} &:= \sum_{j=1}^n  \frac{\partial T_{i,j}}{ \partial {x_i}} \, ,
    \end{aligned}
$$
the first being the (tensor-valued) symmetric gradient  and the second the (vector-valued) tensor divergence.
Then, the Reissner-Mindlin model for a thick plate reads~\cite{mindlin1951influence}:
\begin{equation}\label{eqn:RM}
    \left \lbrace \begin{aligned}
        & \rho h \, \frac{\partial^2 w}{\partial t^2} = \diver( k G h \, (\grad(w) - \bm{\phi})) \, , \\
        &\frac{\rho h^3}{12} \, \frac{\partial^2\bm{\phi}}{\partial t^2} = \Diver( \mathbb{D} \, \Grad(\bm{\phi})) + k G h \, (\grad(w) - \bm{\phi}) \, ,
    \end{aligned} \right.
\end{equation}
with $w$ the vertical displacement, $\bm{\phi} = (\phi_1,\phi_2)$ the deflection of the vertical cross-section along the $x$ and $y$ axis respectively, $\rho$ the surface density of the plate, $h$ the thickness of  the plate, $G$ the shear modulus and $k$ the correction factor~\cite{brugnoli2019portI}.

\begin{remark}
    Such class of mechanical systems has been thoroughly studied in~\cite{ponce2024systematic}, where a method is proposed to derive --given a set of kinematic assumptions-- from the general linear elasticity problem of an $N$-dimensional system, a reduced pH model.
\end{remark}

This model allows us to describe a 2D plate under shear and bending constraints. Moreover it can be simplified to the Kirchhoff-Love plate model when neglecting the angular momentum $\partial_t^2 \bm{\phi} = 0$ and the shear deformation $ \grad(w) = \bm{\phi}$. Let us now write this system using firstly a Stokes-Dirac structure, then a Stokes-Lagrange structure.

Such a comparison regarding the position of the operators (in the Hamiltonian or in the structure matrix) has been explored in~\cite{SchoberlSiukaECC2013} using jet bundles and field theories, here we will study it in the Stokes-Lagrange formalism and additionally show that the obtained energy boundary ports can be put in the Stokes-Lagrange structure and that we can reduce the Reissner-Minldin plate model to the Kirchhoff-Love model through an added constraint.

\subsubsection{Stokes-Dirac formulation} 
\paragraph{Reissner-Mindlin} Following~\cite{brugnoli2019portI,brugnoli2019portII}, let us firstly define the Hamiltonian and choose the energy variables:
$$
\begin{aligned}
    H_{RM} := \frac{1}{2}\int_\Omega & \Big( \rho h (\partial_t w)^2 + \frac{\rho h^3}{12}|\partial_t \bm{\phi}|^2 + kGh \, | \grad(w) - \bm{\phi}|^2  \\
    & + \Grad(\bm{\phi}) : \mathbb{D}\,  \Grad(\bm{\phi}) \, \Big) \, \d x \, .
\end{aligned}
$$
Let us now select the 4 following quantities as energy variables: 
$$ \alpha_{RM}^{SD} = \begin{bmatrix}
    p_w \\ \tensor{\varepsilon}_{\bm{\phi}}\\ \bm{p}_{\bm{\phi}}\\ \bm{\varepsilon_{w\phi}}
\end{bmatrix}:= \begin{bmatrix}
     \rho h \partial_t w \\ \Grad(\bm{\phi}) \\ \frac{\rho h^3}{12}\partial_t \bm{\phi} \\ \grad(w) - \bm{\phi}
\end{bmatrix},$$
namely the linear momentum, the curvature tensor, the angular momentum and shear deformation, respectively.
Then, the Hamiltonian can be written as:
$$
\begin{aligned}
    H_{RM}^{SD} := \frac{1}{2}\int_\Omega & \frac{1}{\rho h}p_w^2 + \frac{12}{\rho h^3}|\bm{p_\phi}|^2 + \tensor{\varepsilon}_{\bm{\phi}}: \mathbb{D} \, \tensor{\varepsilon}_{\bm{\phi}} + kGh \, |\bm{\varepsilon_{w,\phi}}|^2 \, \d x \, .
\end{aligned}
$$
The dynamics~\eqref{eqn:RM} reads then:
\begin{equation} \label{eqn:reissner-mindlin-stokes-dirac}
    \begin{bmatrix}
        \partial_t p_w \\
        \partial_t \tensor{\varepsilon}_{\bm{\phi}} \\
        \partial_t \bm{p}_{\bm{\phi}} \\ 
        \partial_t \bm{\varepsilon_{w\phi}}
    \end{bmatrix} = \begin{bmatrix}
        0 & 0 & 0 & \diver \\
         0 & 0 & \Grad  & 0\\
         0 & \Diver & 0 & I_2 \\
         \grad & 0 & -I_2 & 0
    \end{bmatrix} \begin{bmatrix}
         v \\ \tensor{\sigma}_\phi \\ \bm{\omega} \\ \bm{N}
    \end{bmatrix},
\end{equation}
together with the constitutive relations:
$$
    \begin{bmatrix}
        v \\ \tensor{\sigma}_\phi \\ \bm{\omega} \\ \bm{N}
    \end{bmatrix} = \begin{bmatrix}
        \frac{1}{\rho h } & 0 & 0 & 0 \\
        0 & \mathbb{D} & 0 & 0 \\
        0 & 0 & \frac{12}{\rho h^3} & 0\\
        0  & 0 & 0 & kGh
    \end{bmatrix} \begin{bmatrix}
         p_w \\
         \tensor{\varepsilon}_{\bm{\phi}} \\
         \bm{p}_{\bm{\phi}} \\ 
        \bm{\varepsilon_{w\phi}}
    \end{bmatrix},
$$
where $v$ is the velocity, $\tensor{\sigma}_\phi$ is the momenta tensor, $\bm{p}_{\bm{\phi}}$ is the angular momentum and $N$ is the shear stress. Finally, the power balance reads:
\begin{theorem}(Reissner-Mindlin power balance - Stokes-Dirac case) \label{eqn:power-balance-reissner-mindlin-dirac}
    $$
        \frac{\rm d}{{\rm d}t}H_{RM}^{SD} = \int_{\partial \Omega}  v \bm{N} \cdot \bm{n} + \bm{\omega}^\top \tensor{\sigma}_\phi \bm{n} \, \d s \, .
    $$
\end{theorem}
\begin{proof}
    The proof is to be found in~\cite[eq. (42)]{brugnoli2019portI}.
\end{proof}

\paragraph{From Reissner-Mindlin to Kirchhoff-Love model}
The Kirchhoff-Love model is obtained by neglecting the angular momentum and assuming pure bending of the plate, \textit{i.e.} $\bm{p_\phi} = 0$ and $\bm{\varepsilon_{w,\phi}} = \grad(w) - \bm{\phi} = 0$, this leads to the following constrained Differential Algebraic Equation (DAE) formulation of the Kirchhoff-Love plate model (see \textit{e.g.}~\cite{Zwart_2024}):
\begin{equation} \label{eqn:kirchhoff-love-stokes-dirac}
    \begin{bmatrix}
        \partial_t p_w \\
        \partial_t \tensor{\varepsilon}_{\bm{\phi}} \\
        \bm{0} \\ \bm{0}
    \end{bmatrix} = \begin{bmatrix}
        0 & 0 & 0 & \diver \\
         0 & 0 & \Grad  & 0\\
         0 & \Diver & 0 & I_2 \\
         \grad & 0 & -I_2 & 0
    \end{bmatrix} \begin{bmatrix} v \\ \tensor{\sigma}_\phi \\ \bm{\omega} \\ \bm{N}
    \end{bmatrix},
\end{equation}
where $\bm{N}$ and $\bm{\omega}$ now play the role of Lagrange multipliers. Moreover, the Hamiltonian is now reduced to:
$$
    H^{SD}_{KL} = \frac{1}{2}\int_\Omega  \frac{1}{\rho h}p_w^2 + \tensor{\varepsilon}_{\bm{\phi}}: \mathbb{D} \, \tensor{\varepsilon}_{\bm{\phi}} \, \d x \, .
$$

\begin{theorem}(Kirchhoff-Love power balance - Stokes-Dirac case)
The power balance of the Kirchhoff-Love plate model written in Stokes-Dirac formulation reads:
    $$
        \frac{\rm d}{{\rm d}t}H_{KL}^{SD} = \int_{\partial \Omega}   v \Diver(\tensor{\sigma}_\phi) \cdot \bm{n}  + \grad(v)^\top \tensor{\sigma}_\phi \bm{n} \, \d s \, .
    $$
\end{theorem}
\begin{proof}
    Let us compute the time derivative of the Hamiltonian $H^{SD}_{KL}$:
    $$
        \begin{aligned}
            \frac{\rm d}{{\rm d}t} H^{SD}_{KL} & = \int_\Omega v \, \partial_t p_w + \tensor{\sigma}_{\bm{\phi}} : \partial_t \tensor{\varepsilon}_{\bm{\phi}} \, \d x \, , \\
            & = \int_\Omega v \, \diver(\bm{N}) + \tensor{\sigma}_{\bm{\phi}} : \Grad(\bm{w}) \, \d x \, ,\\
            & = \int_\Omega - \grad(v) \cdot \bm{N} - \Diver(\tensor{\sigma}_{\bm{\phi}}) \cdot \bm{w} \, \d x 
             + \int_{\partial \Omega} v \bm{N} \cdot \bm{n} + \bm{w}^\top \tensor{\sigma}_{\bm{\phi}} \bm{n} \, \d s \, , \\
            & = \int_\Omega - \bm{w} \cdot \bm{N} + \bm{N} \cdot \bm{w} \, \d x 
             + \int_{\partial \Omega} v \Diver(\tensor{\sigma}_{\bm{\phi}}) \cdot \bm{n} + \grad(v)^\top \tensor{\sigma}_{\bm{\phi}} \bm{n} \, \d s \, , \\
            & = \int_{\partial \Omega} v \Diver(\tensor{\sigma}_{\bm{\phi}}) \cdot \bm{n} + \grad(v)^\top \tensor{\sigma}_{\bm{\phi}} \bm{n} \, \d s \, ,
        \end{aligned}
    $$
    which ends the proof.
\end{proof}
\begin{remark}
Note that this model uses first order differential operators only, contrarily to the classical representation of the Kirchhoff-Love model, which requires second order differential operators;
the following unconstrained system is obtained when solving the constraints analytically:
$$
    \begin{bmatrix}
        \partial_t p_w \\
        \partial_t \tensor{\varepsilon}_{\bm{\phi}}
    \end{bmatrix} = \begin{bmatrix}
        0 & - \diver(\Diver(\cdot)) \\
       \Grad(\grad(\cdot)) &0
    \end{bmatrix} \begin{bmatrix}
        v \\
        \tensor{\sigma_\phi}
    \end{bmatrix},
$$
together with the constitutive relations $v = \frac{1}{\rho h} p_w$, $\tensor{\sigma}_{\bm{\phi}} = \mathbb{D} \, \tensor{\varepsilon}_{\bm{\phi}}$; such a formulation has been studied in~\cite{brugnoli2019portII}.
\end{remark}

\subsubsection{Stokes-Lagrange formulation}
\paragraph{Reissner-Mindlin}
Let us now choose a different set of energy variables, namely:
$$
\alpha^{SL}_{RM} := \begin{bmatrix}
    w & p_w & \bm{\phi}^\top & \bm{p_\phi}^\top
\end{bmatrix}^\top,
$$
then the Hamiltonian becomes: 
$$
    \begin{aligned}
    H_{RM}^{SL} := \frac{1}{2}\int_\Omega & \Big( \frac{1}{\rho h} p_w^2 + \frac{12}{\rho h^3}|\bm{p_\phi}|^2 + kGh \, | \grad(w) - \bm{\phi}|^2  \\
    & +  \Grad(\bm{\phi}) : \mathbb{D}\,  \Grad(\bm{\phi}) \, \Big) \, \d x \, .
    \end{aligned}
$$
In order to define the system properly, let us compute the variational derivatives with respect to each variable.
\begin{lemma}\label{lem:var-der-RM-SL}
    The variational derivatives of $H_{RM}^{SL}$ are given as:
    $$
        \left\{ \begin{aligned}
        \delta_w H_{RM}^{SL} &= - \diver( kGh\, (\grad(w) - \bm{\phi})) \, , \\
        \delta_{p_w} H_{RM}^{SL} &= \frac{1}{\rho h} p_w \, ,\\
        \delta_{\bm{\phi}} H_{RM}^{SL} &= kGh \, (\grad(w) - \bm{\phi})) - \Diver(\mathbb{D} \, \Grad(\bm{\phi})) \, , \\
        \delta_{\bm{p_\phi}} H_{RM}^{SL} &= \frac{12}{\rho h^3} \bm{p}_{\bm{\phi}} \, .
        \end{aligned} \right.
    $$
\end{lemma}
\begin{proof}
    The proof is given in~\ref{apx-subsec:reissner-mindlin-varder-lagrange}.
\end{proof}
Now, we can write the dynamics as:
\begin{equation} \label{eqn:reissner-mindlin-stokes-lagrange}
    \begin{bmatrix}
        \partial_t w \\
        \partial_t p_w \\
        \partial_t \bm{\phi} \\
        \partial_t \bm{p_\phi}
    \end{bmatrix} = \begin{bmatrix}
        0 & 1 & 0 & 0\\
        -1 & 0 & 0 & 0\\
        0 & 0 & 0 & 1\\
        0 & 0 & -1& 0
    \end{bmatrix} \begin{bmatrix}
        \delta_w H_{RM}^{SL} \\
        \delta_{p_w} H_{RM}^{SL} \\
        \delta_{\bm{\phi}} H_{RM}^{SL} \\
        \delta_{\bm{p_\phi}} H_{RM}^{SL}
    \end{bmatrix},
\end{equation}
and with the Lagrange structure matrix:
$$
    S_{RM}^{SL} := \begin{bmatrix}
        - \diver ( kGh \, \grad( \cdot ) ) & 0 & - \diver ( kGh \, \cdot)  &  0\\
        0 & \frac{1}{\rho h} & 0 & 0\\
        kGh \, \grad(\cdot) & 0 &  -\Diver( \mathbb{D} \, \Grad(\cdot)) + kGh  & 0\\
        0 & 0 & 0 & \frac{12}{\rho h^3}
    \end{bmatrix},
$$
the constitutive relations of Lemma~\ref{lem:var-der-RM-SL} read:
$$
e_{RM}^{SL} = S_{RM}^{SL} \, \alpha^{SL}_{RM} \, ,
$$ 
where $e_{RM}^{SL} := \begin{bmatrix}
        \delta_w H_{RM}^{SL} &
        \delta_{p_w} H_{RM}^{SL} &
        \delta_{\bm{\phi}} H_{RM}^{SL}\,^\top &
        \delta_{\bm{p_\phi}} H_{RM}^{SL}\,^\top
    \end{bmatrix}^\top$. 
This system has the following power balance:
\begin{theorem}(Reissner-Mindlin power balance - Stokes-Lagrange case)
The power balance of the Reissner-Mindlin plate model written in Stokes-Lagrange formulation reads:
    $$
        \frac{\rm d}{{\rm d}t} H^{SL}_{RM} = \int_{\partial \Omega} \underbrace{(\partial_t w)}_{=: \partial_t \chi^\partial_{RM,1}} \,  \underbrace{kGh \,(\grad(w) - \phi)\cdot \bm{n}}_{\varepsilon^\partial_{RM,1}} + \underbrace{\partial_t \bm{\phi}^\top}_{=: \partial_t \chi^\partial_{RM,2 }} \, \underbrace{\mathbb{D} \, \Grad(\bm{\phi}) \bm{n}}_{\varepsilon^\partial_{RM,2}} \, \d x \, .
    $$
\end{theorem}
\begin{remark}
    Similarly to Remark~\ref{rmk:power-balance-wave-equation}, this power balance has the same physical interpretation  as in the Stokes-Dirac case (Theorem~\ref{eqn:power-balance-reissner-mindlin-dirac}). In particular one can identify the energy ports $\chi^\partial_{RM}, \varepsilon^\partial_{RM}$ with the power ports of the Stokes-Dirac case as: $\partial_t \chi^\partial_{RM,1} = v|_{\partial \Omega}$, $\partial_t \chi^\partial_{RM,2} = \bm{w}|_{\partial \Omega} $, $\varepsilon^\partial_{RM,1} = \bm{N}|_{\partial \Omega} \cdot \bm{n}$, and $\varepsilon^\partial_{RM,2} = \tensor{\sigma}_{\bm{\phi}} |_{\partial \Omega} \, \bm{n}$.
\end{remark}

\paragraph{From Reissner-Mindlin to Kirchhoff-Love model}
As before, let us apply the constraints $p_\phi = 0$ and $\grad\, w = \phi$, leading to:
\begin{equation}\label{eqn:kirchhoff-love-stokes-lagrange}
\begin{bmatrix}
    \Id & 0 & 0 & 0\\
    0 & \Id & 0 & 0\\
    \grad & 0 & 0 & 0\\
    0 & 0 & 0 & \bm{0}
\end{bmatrix}
    \begin{bmatrix}
        \partial_t w \\
        \partial_t p_w \\
        \bm{0} \\
        \bm{0}
    \end{bmatrix} = \begin{bmatrix}
        0 & 1 & 0 & 0\\
        -1 & 0 & 0 & 0\\
        0 & 0 & 0 & 1\\
        0 & 0 & -1& 0
    \end{bmatrix} \begin{bmatrix}
        \delta_w H_{KL}^{SL} \\
        \delta_{p_w} H_{KL}^{SL} \\
        \bm{N}_{\bm{\phi}}\\
        \bm{w}
    \end{bmatrix},
\end{equation}
where $\bm{N}_{\bm{\phi}}$ and  $\bm{w}$ play the role of Lagrange multipliers of the constraints. This allows us to define the constrained Hamiltonian as:
$$
     H_{RM}^{SL} := \frac{1}{2}\int_\Omega \frac{1}{\rho h} p_w^2  +  \Grad(\grad(w)) : \mathbb{D}\,  \Grad(\grad(w)) \, \d x \, ,
$$
and deduce the reduced power balance:
\begin{theorem}(Kirchhoff-Love power balance - Stokes-Lagrange case)
    $$
        \frac{\rm d}{{\rm d}t} H^{SL}_{RM} = \int_{\partial \Omega} \frac{\partial}{\partial t}\left[ \grad(w) \right]^\top \, \mathbb{D} \, \Grad(\grad(w)) \bm{n} \, \d s \, .
    $$
\end{theorem}
\begin{proof}
    Let us compute the time derivative of the Hamiltonian:
    $$
        \begin{aligned}
            \frac{\rm d}{{\rm d}t} H^{SL}_{KL} & = \int_\Omega \frac{1}{\rho h} p_w \, \partial_t p_w + \Grad(\grad(\frac{\partial}{\partial_t} w)) : \mathbb{D}\,  \Grad(\grad(w)) \, \d x \, , \\
            & = \int_\Omega \delta_{p_w} H^{SL}_{KL} \cdot \delta_w H^{SL}_{KL} - \delta_{w}H_{KL}^{SL} \cdot \delta_{p_w} H^{SL}_{KL} \, \d x \\
            & \qquad +   \int_{\partial \Omega} \frac{\partial}{\partial t}(\grad(w)) ^\top \mathbb{D}\,  \Grad(\grad(\partial_t w))\, \d s \, , \\
            &= \int_{\partial \Omega} \frac{\partial}{\partial t} (\grad(w)) ^\top \mathbb{D}\,  \Grad(\grad(\partial_t w))\, \d s \, .
        \end{aligned}
    $$
\end{proof}
Note that this is an example of \emph{constrained} Stokes-Lagrange structure.
\begin{remark}
    In this particular case, the two bottom lines of~\eqref{eqn:kirchhoff-love-stokes-lagrange} can be removed directly and yield a reduced system with the two state variables $w$ and $p_w$. Here they are kept as they will be useful when studying the relationships between the Stokes-Dirac and Stokes-Lagrange representations in Subsection~\ref{subsec:passing-representation-plate}.
\end{remark}

\subsubsection{Passing from one representation to the other} \label{subsec:passing-representation-plate}
As in Subsection~\ref{subsec:operator-transposition}, let us define the operator $G$ that will allow us to transform one representation into the other. Such an operator is given as:
$$
    G := \begin{bmatrix}
        0 & \Id & 0 & 0\\
        0 & 0 & \Grad & 0\\
        0 & 0 & 0 & \Id \\
        \grad & 0 & - \Id & 0  
    \end{bmatrix}, \qquad G^\dag = \begin{bmatrix}
        0 & 0 & 0 & -\diver \\
        \Id & 0 & 0 & 0\\
        0 & - \Diver & 0 & -\Id \\
        0 & 0 &  \Id & 0  
    \end{bmatrix}.
$$
Applying Theorem~\ref{thm:dirac-lagrange-equiv} shows the following:
\begin{theorem} The following diagram commutes:
    \begin{figure}[ht]
	\centering
	
	\begin{tikzpicture}
		\draw node at (-4,0) {\eqref{eqn:reissner-mindlin-stokes-dirac}};
		\draw node at (4,0) {\eqref{eqn:reissner-mindlin-stokes-lagrange}};
		\draw node at (-4,-2) {\eqref{eqn:kirchhoff-love-stokes-dirac}};
		\draw node at (4,-2) {\eqref{eqn:kirchhoff-love-stokes-lagrange}};

		\draw node at (-5.5,-1) {$\begin{bmatrix}
		    \bm{p_\phi} = \bm{0} \\ \bm{\varepsilon_{w,\phi}}=\bm{0} 
		\end{bmatrix}$};
		\draw node at (5.5,-1) {$\begin{bmatrix}
		    \bm{p_\phi} = \bm{0} \\ \grad \, w = \bm{\phi}
		\end{bmatrix}$} ;

		\draw node at (0,.5) {$G, G^\dag$};
		\draw node at (0,-2.5) {$G,G^\dag$};

		\draw[<->,>=latex] (-2.5,0) --(2.5,0);
		\draw[<->,>=latex] (-2.5,-2) --(2.5,-2);
		
		\draw[->,>=latex] (-4,-.5) --(-4,-1.5);
		\draw[->,>=latex] (4,-.5) --(4,-1.5);
	\end{tikzpicture}
 \caption{Commutative diagram showing the relations between the representations of the Reissner-Mindlin and Kirchhoff-Love plate models .}
 \label{fig:Plate-com-diag}
\end{figure}
\end{theorem}
\subsection{Maxwell's equations}
Let us now write the Maxwell's equations~\cite{kovetz2000electromagnetic} in Stokes-Dirac and Stokes-Lagrange representations. In this particular example, the former is the classical representation of the Maxwell's equation, whereas the latter coincides with the  \emph{vector potential formulation}.
\subsubsection{Stokes-Dirac formulation}
Let us firstly write the classical port-Hamiltonian representation of the Maxwell's equations~\cite{van2002hamiltonian}:
\begin{equation}  \label{eqn:maxwell-dirac}
	\begin{bmatrix}
		\partial_t \bm{D} \\
		\partial_t \bm{B}
	\end{bmatrix} = \begin{bmatrix}
		0 & \curl \\ - \curl & 0
	\end{bmatrix}
	\begin{bmatrix}
		\bm{E} \\ \bm{H}
	\end{bmatrix} - \begin{bmatrix}
		\bm{J} \\ 0
	\end{bmatrix},
\end{equation}
with $\bm{D}$ the electric field, $\bm{E}$ the electric flux, $\bm{B}$ the magnetic field, $\bm{H}$ the magnetic flux and $\bm{J}$ the current density. The Hamiltonian reads:
$$
    H_{EM} = \frac{1}{2}\int_\Omega \frac{1}{\varepsilon_0} |\bm{D}|^2 + \frac{1}{\mu_0} |\bm{B}|^2 \, \d x \, ,
$$
with $\mu_0>0$ the magnetic permeability, $\varepsilon_0>0$ and electric permittivity, along with the constitutive relations:
$$
\left \lbrace
	\begin{aligned}
		\bm{E} = \delta_{\bm{D}} H_{EM} = \frac{1}{\varepsilon_0} \bm{D} \, , \qquad \diver(\bm{D}) = \rho \, , \\
		\bm{H} = \delta_{\bm{B}} H_{EM} = \frac{1}{\mu_0} \bm{B} \, , \qquad         \diver(\bm{B}) = 0 \, .	     
	\end{aligned} \right.
$$
In the following we will assume that the charge density $\rho$ is zero.
\begin{remark}
    Note that the gauge conditions $ \diver(\bm{D}) = 0 = \diver(\bm{B}) $ lie in  the Stokes-Lagrange structure of the system, however for the sake of simplicity, they will be dropped in the following as they do not play a role in the dynamics of the system. Indeed $\diver \, \curl = 0$, hence $\frac{\partial}{\partial t}\diver(\bm{D}) = \diver(\curl(\bm{H})) = 0$.
\end{remark}
Let us now define the Lagrange structure operators as: 
$$
P^{SD} =\begin{bmatrix}
    I_3 & 0\\
    0 & I_3
\end{bmatrix} , \qquad S^{SD} = \begin{bmatrix}
	\frac{1}{\varepsilon_0} & 0 \\ 0 & \frac{1}{\mu_0}
\end{bmatrix},
$$ 
the Dirac structure operator:
$$
J^{SD} := \begin{bmatrix}
	0 & \curl \\ - \curl & 0
\end{bmatrix},
$$ 
the control operator:
$$
B^{SD} = \begin{bmatrix}
-I_3 & 0
\end{bmatrix}^\top,
$$ 
and the \emph{distributed} control $u^{SD} = \bm{J}$. In the following, we denote the state and co-state variables of this system as $ \alpha^{SD} = \begin{bmatrix} 
	\bm{D} & \bm{B}
\end{bmatrix} ^\top  $ and $e^{SD}= \begin{bmatrix}
\bm{E} & \bm{H}
\end{bmatrix} ^\top$. The system now reads:
$$
	\begin{cases}
		\partial_t P^{E} \alpha^{SD} = J^{SD} e^{SD} + B^{SD} u^{SD} \, , \\
		e^{SD} = S^{E} \alpha^{SD} \, .
	\end{cases}
$$
Finally, the power balance is:

\begin{theorem}(Maxwell's equations power balance - Stokes-Dirac case)
$$
    \frac{\rm d}{{\rm d}t}H_{EM} = \int_{\partial \Omega} (\bm{H} \times \bm{E}) \cdot \bm{n}\, \d s - \int_\Omega \bm{J} \cdot \bm{E} \, \d x \, .
$$
\end{theorem}
\begin{remark}
    Defining the Poynting vector as $\bm{\Pi} := \bm{E} \times \bm{H}$, one can rewrite the power balance as:
$$
    \frac{\rm d}{{\rm d}t}H_{EM} = - \int_{\partial \Omega} \bm{\Pi} \cdot \bm{n} \, \d s - \int_\Omega \bm{J} \cdot \bm{E} \, \d x \, ,
$$
which shows that such a vector corresponds to the energy flux through the boundary of the considered domain.
\end{remark}
\begin{remark}
    Note that $\bm{J}$ is usually decomposed in two parts, $\bm{J} = \bm{J}_u + \bm{J}_r$, with $\bm{J}_u$ the controlled current and $\bm{J}_r = \sigma \bm{E},$ being the induced current, with $\sigma \geq0$ the conductivity, leading to the additional dissipative port $(\bm{E}, \bm{J}_r)$. With such a distributed dissipation, the power balance more precisely reads:
$$
\begin{aligned}
    \frac{\rm d}{{\rm d}t}H_{EM} &= \int_{\partial \Omega} (\bm{H} \times \bm{E}) \cdot \bm{n} \, \d s - \int_\Omega \bm{J}_u \cdot \bm{E} \, \d x - \int_\Omega \sigma |\bm{E}|^2 \, \d s \, , \\
    &\leq \int_{\partial \Omega} (\bm{H} \times \bm{E}) \cdot \bm{n} \, \d s - \int_\Omega \bm{J}_u \cdot \bm{E} \, \d x \, .
\end{aligned}
$$
\end{remark}

\subsubsection{Stokes-Lagrange representation}

Let us now write the system using the Stokes-Lagrange representation. It turns out that such a representation exactly coincides with the vector potential formulation of the Maxwell's equations. To construct it, let us assume that $\Omega$ is simply connected, then provided that $\diver(B) = 0$, there exists an $\bm{A}$ such that:
$$
        \curl(\bm{A}) = \bm{B} \, ,
$$
which allows us to rewrite the Hamiltonian as:
$$
    H_{EM}^{SL} = \int_\Omega \frac{1}{\mu_0} |\curl(A)|^2 + \frac{1}{\varepsilon_0} | \bm{D}|^2 \, \d x \, ,
$$
and the system can be written using a Stokes-Lagrange formulation as:
\begin{equation} \label{eqn:maxwell-lagrange}
	\begin{bmatrix}
		\partial_t \bm{D} \\
		 \partial_t  \bm{A}
	\end{bmatrix} = \begin{bmatrix}
		0 &I_3 \\ - I_3 & 0
	\end{bmatrix}
	\begin{bmatrix}
		\bm{E} \\ \delta_{\bm{A}} H
	\end{bmatrix} - \begin{bmatrix}
		\bm{J} \\ 0
	\end{bmatrix}.
\end{equation}
 Let us note $J^{SL} = \begin{bmatrix}
	0 & I_3 \\
	- I_3 & 0
\end{bmatrix}$, $ P^{SL} = \begin{bmatrix}
I_3 & 0 \\ 0 & 1I_3
\end{bmatrix}$, $ S^{SL} = \begin{bmatrix}
I_3 & 0 \\ 0 & \curl ( \frac{1}{\mu_0} \curl\cdot  )
\end{bmatrix}$, and $ B^{SL} = \begin{bmatrix}
-I_3 & 0
\end{bmatrix}$.
We will denote the associated state and effort variables as: $\alpha^{SL} = \begin{bmatrix}
	\bm{D} & \bm{A}
\end{bmatrix} ^ \top$ and $e^{SL} = \begin{bmatrix}
\delta_{\bm{D}} H_{EM}^{SL}&\delta_{\bm{A}} H_{EM}^{SL}
\end{bmatrix}  ^ \top $. 
The Stokes-Lagrange formulation now reads:
$$
    \begin{cases}
		\partial_t P^{I} \alpha^{SL} = J^{SL} e^{SL} + B^{I} u \, , \\
		e^{SL} = S^{I} \alpha^{SL} \, .
	\end{cases}
$$
Finally, the power balance of this system is:
\begin{theorem}(Maxwell's equations power balance - Stokes-Lagrange case)
    $$
        \frac{\rm d}{{\rm d}t} H_{EM}^{SL} = \int_{\partial \Omega} \frac{1}{\mu_0}\curl(\bm A) \times \partial_t \bm A \cdot \bm{n} \, \d s - \int_{\partial \Omega} \bm{J} \cdot \bm{E}\, \d s \, .
    $$
\end{theorem}
\subsubsection{Passing from one representation to the other}
Following Subsection~\ref{subsec:operator-transposition}, let us define $G := \begin{bmatrix}
1& 0 \\ 0 & \curl
\end{bmatrix}$ an operator from $L^2(\Omega, \mathbb{R}^N) \times H^\curl(\Omega)$ to $L^2(\Omega, \mathbb{R}^N) \times L^2(\Omega, \mathbb{R}^N)$. This operator is \textbf{not} invertible in general but by restricting ourselves to the subspace $ \mathcal{\bm{K}} = \{ \bm{B} \in H^\curl(\Omega), \diver(\bm{B})  = 0\}$ and assuming that the domain $\Omega$ is simply connected, allows us to consider $\curl$ as invertible~\cite{arnold2018finite}. Then:
$$
    \begin{cases}
		G \alpha ^{SL} = \alpha ^{SD} \, , \\
		e^{SL} = G^\dag e^{SD} \, .
	\end{cases}
$$
And we have the following relations:
$$ 
J^{SD} = G J^{SL} G^\dag \, , \qquad S^{SL} = G S ^{SD} G^\dag \, , \qquad P^{SL} = G P^{SD} G^{-1} \, .
$$

\subsubsection{Low-frequency approximation}

The low-frequency approximation (LF) is a usual assumption when considering technical applications~\cite{buffa2000justification} (\textit{e.g.} electrical machines). It consists in neglecting the displacement currents $\partial_t \bm D = 0$, the Stokes-Dirac representation then reads:
\begin{equation} \label{eqn:maxwell-dirac-lowfrequency}
        \begin{bmatrix}
            \bm{0}\\
            \partial_t \bm{B}
        \end{bmatrix} = \begin{bmatrix}
            0 & \curl \\ \curl & 0
        \end{bmatrix} \begin{bmatrix}
            \bm{E} \\ \bm{H}
        \end{bmatrix} - \begin{bmatrix}
            \bm{J} \\ 0
        \end{bmatrix},
\end{equation}
and the Stokes-Lagrange representation:
\begin{equation} \label{eqn:maxwell-lagrange-lowfrequency}
    \begin{bmatrix}
            \bm{0} \\
            \partial_t \bm{A}
        \end{bmatrix} = \begin{bmatrix}
            0 & 1 \\ -1 &0
        \end{bmatrix} \begin{bmatrix}
            \bm{E} \\ \delta_{\bm{A}}H^{SL}
        \end{bmatrix} - \begin{bmatrix}
            \bm{J} \\ 0
        \end{bmatrix}.
\end{equation}
\begin{remark}
    This representation has been studied in~\cite{reis2023passivity} when setting $\bm{J} = \sigma \bm{E}$ and solving the constraint $\delta_{\bm{A}} H^{SL} - \bm{J} = 0$, yielding the diffusion equation $\sigma \partial_t \bm{A} - \curl(\frac{1}{\mu} \curl(\bm{A})) = 0$.
\end{remark}

\begin{theorem}
    Denoting by $\Pi= \text{Diag} \begin{bmatrix}
    \bm{0} & \Id
\end{bmatrix}$ the projection operator, one gets the commutative diagram:
\begin{figure}[ht]
	\centering
	
	\begin{tikzpicture}
		\draw node at (-4,0) {Stokes-Dirac \eqref{eqn:maxwell-dirac}};
		\draw node at (4,0) {Stokes-Lagrange \eqref{eqn:maxwell-lagrange}};
		\draw node at (-4,-2) {LF Stokes-Dirac \eqref{eqn:maxwell-dirac-lowfrequency}};
		\draw node at (4,-2) {LF Stokes-Lagrange \eqref{eqn:maxwell-lagrange-lowfrequency}};

		\draw node at (-4.5,-1) {$\Pi$};
		\draw node at (4.5,-1) {$\Pi$};

		\draw node at (0,.5) {$G$};
		\draw node at (0,-2.5) {$G$};

		\draw[<->,>=latex] (-1,0) --(1,0);
		\draw[<->,>=latex] (-1,-2) --(1,-2);
		
		\draw[->,>=latex] (-4,-.5) --(-4,-1.5);
		\draw[->,>=latex] (4,-.5) --(4,-1.5);
	\end{tikzpicture}
 \caption{Commutative diagram showing the relations between the two representations of the Maxwell's equation and their corresponding LF approximations.}
 \label{fig:Maxwell-com-diag}
\end{figure}
\end{theorem}

\subsubsection{Maxwell's equations on a 2D domain}

The vector potential formulation of the Maxwell's equation is particularly useful when considering 2D domains~\cite[Chap. 9]{assous2018mathematical} (\textit{e.g.} a rotor/stator setup). Let us consider the electromagnetic field to lie parallel to the horizontal plane $\bm{B} = \begin{bmatrix}
    B_x & B_y &0
\end{bmatrix}^\top$ and the electric field to be vertical $\bm{D} = \begin{bmatrix}
    0 & 0 & D_z
\end{bmatrix}^\top$, denoting the reduced operators:
$$
    \begin{cases}
        \grad^\perp (v) := \begin{bmatrix}
            - \partial_x v \\
            \partial_y v
        \end{bmatrix},  \\
        \curldeuxD (\bm{w}) := \partial_x w_y - \partial_y w_x \, ,
    \end{cases}
$$
along with the reduced variable $\bm{B} = \begin{bmatrix}
    B_x & B_y
\end{bmatrix}^\top$, $E = E_z$, the Maxwell's equations in Transverse Electric (TE) mode read:
$$
    \begin{bmatrix}
        \partial_t D \\
        \partial_t \bm{B}
    \end{bmatrix} = \begin{bmatrix}
        0 & \curldeuxD \\
        - \grad^\perp & 0
    \end{bmatrix} \begin{bmatrix}
        E \\ \bm{H}
    \end{bmatrix} - \begin{bmatrix}
        J \\ 0
    \end{bmatrix}.
$$
When the vector potential formulation is considered, we get: $\bm{B} = \grad^\perp A$, hence the vector potential $A$ is now a scalar field instead of a vector field; and the vector potential formulation reads:
$$
    \begin{bmatrix}
        \partial_t D \\
        \partial_t A
    \end{bmatrix} = \begin{bmatrix}
        0 & 1 \\
        -1 & 0
    \end{bmatrix} \begin{bmatrix}
        E \\ \delta_A H
    \end{bmatrix} - \begin{bmatrix}
        J  \\ 0
    \end{bmatrix},
$$
with $\delta_AH = \curldeuxD \left( \frac{1}{\mu} \grad^\perp A \right)$.
\begin{remark}
    Even if these representations are equivalent at the continuous level, they might not be from a modelling perspective. When discretizing the system; one requires only scalar-valued functions in the vector potential formulation, compared to vector-valued function in the classical case, which reduces memory usage (see~\cite{cardoso2024port,haine2021incompressible} for a similar approach applied to fluid dynamics).
\end{remark}

\subsection{Dzektser equation}

As a last example, let us present a nonlocal model, namely the Dzektser equation~\cite{Dzektser72}. Let us consider a $2D$ domain $\Omega$ and $h(t,x,y)$ the (vertical) hydraulic head defined at each point of the domain, the Dzektser equation describes the evolution of the hydraulic head $h$ of underground water while neglecting  momentum, yielding a diffusion-like equation. Contrarily to the Boussinesq equation, it takes into account second order term when modelling the water column which gives us a nonlocal term in the equation. This gives us a linear system depending on two damping parameters $a,b \geq0$ and a parameter of \emph{nonlocality} $\mu \geq 0$, that reads:
$$
(1-\mu \Delta) \partial_t h = a \Delta h - b \Delta^2 h \, ,
$$
and the pH formulation can be formulated as~\cite{bendimerad2023implicit}:
$$
    \begin{bmatrix}
        1 - \mu \Delta & 0  & 0\\
        0 & 1  & 0\\
        0 & 0 & 1
    \end{bmatrix} \begin{bmatrix}
        \partial_t h \\
        \bm{F}_\nabla \\ F_\Delta
    \end{bmatrix} = \begin{bmatrix}
        0 & \diver & \Delta \\ \grad & 0 & 0\\
        -\Delta & 0 & 0
    \end{bmatrix} \begin{bmatrix}
        h \\ \bm{E}_\nabla \\ E_\Delta
    \end{bmatrix}, \quad
    \begin{cases}
        \bm{E}_\nabla = a \bm{F}_\nabla \, , \\
        E_\Delta = b F_\Delta \, ,
    \end{cases}
$$
with $(\bm{F}_\nabla, \bm{E}_\nabla)$, $(F_\Delta, E_\Delta)$ the two dissipative ports. The associated Hamiltonian is:
$$
H_D = \frac{1}{2} \int_\Omega h^2 + \mu |\grad h |^2 \, \d x \, .
$$
Note that the operator $(1 - \mu \Delta)$ with $\Delta = \frac{\partial^2}{\partial x^2} + \frac{\partial^2}{\partial y^2}$ from $H^1_0(\Omega)$ to $L^2$ is invertible; in particular it does not introduce constraints to the dynamics.

Then, one can define the Stokes-Lagrange operators as, $P = 1 - \mu \Delta$, $S = \Id$, and the Hamiltonian reads:
$$
    H_D = \frac{1}{2} \int_\Omega h^2 + \mu |\grad h |^2 \, \d x = \frac{1}{2} \langle S\,h, P\,h \rangle + \frac{1}{2} \int_{\partial_\Omega} \underbrace{h}_{=: \gamma(h)} \underbrace{\mu \grad h \cdot n}_{ =:\beta(h)} \, \d s \, .  
$$

\section{Application to control systems.} \label{sec:control}

One of the key features of port-Hamiltonian systems is the passivity of the power control port. More precisely, the natural boundary controls of a system are obtained (\textit{e.g.} force/velocity, flow/pressure) and allow the use of passivity to design controllers and interconnect pH system through their Dirac structures~\cite{ortega2008control}. In a recent work~\cite{borja2023interconnection}, it was shown that another type of boundary control could be useful; those are defined as the time integral of a passive output or time derivative of a passive output. Let us show that the energy control ports allow us to design such interconnection and that choosing between the two proposed Hamiltonians allows for choosing between the time integral of the passive output or the time derivative of the passive input. As a result, one can interconnect the two systems using a non-separable Hamiltonian (see Definition~\ref{def:separable-hamiltonian}). Such property is valuable as the usual power port interconnection of two pH systems yields an interconnected pH system whose Hamiltonian is the sum of the Hamiltonian of the two subsystem; adding a non separable term allows for a broader set of interconnections.

\begin{definition}\label{def:separable-hamiltonian}(Separable Hamiltonian) Given a latent state $z = \begin{bmatrix}
    z_1 & z_2
\end{bmatrix}^\top \in \mathcal{Z} = \mathcal{Z}_1 \times \mathcal{Z}_2$, a Hamiltonian $H: \mathcal{Z} \rightarrow \mathbb{R}$ is called \emph{separable}, if there exist $H_1:\mathcal{Z}_1 \rightarrow \mathbb{R}$ and $H_2:\mathcal{Z}_2 \rightarrow \mathbb{R}$ (corresponding to the Hamiltonian functionals of the two subsystems), such that:
$$ H(z_1,z_2) = H_1(z_1) + H_2(z_2).$$
In particular, $\delta_zH(z) = 0$ if and only if $\delta_{z_1}H_1(z_1) = 0$ and $\delta_{z_2}H_2(z_2) = 0$.
\end{definition}

Firstly, we will recall the interconnection of two pH systems through power ports, then show how this can be performed using energy ports. The latter has already been considered in the recent work~\cite{schaft2024boundary}; here we will extend it to both methods proposed in~\cite{borja2023interconnection}.

\subsection{Stokes-Dirac interconnection}

Let us take two port-Hamiltonian systems implicitly defined on a Stokes-Dirac structure as follows:
$$
    \begin{bmatrix}
        f^i\\ y^i \\e^i \\u^i
    \end{bmatrix} \in \mathcal{D}^i \subset \mathcal{F}^i \times \mathcal{Y} \times \mathcal{E}^i \times \mathcal{U} \, , \qquad 
    f^i = - \partial_t \alpha^i \, , \qquad 
    e^i = \delta_{\alpha^i} H^i(\alpha^i) \, ,
$$
for $i \in \{1,2\}$. Note that the control $\mathcal{U}$ and observation $\mathcal{Y}$ spaces are assumed to be identical. Then, assuming an interconnection of the form:
$$
    \begin{bmatrix}
        u^1 \\u^2
    \end{bmatrix} = \begin{bmatrix}
        0 & \Id \\ - \Id & 0
    \end{bmatrix} \begin{bmatrix}
        y^1 \\ y^2
    \end{bmatrix} + \begin{bmatrix}
        \nu^1 \\ \nu^2
    \end{bmatrix},
$$
the system is (cyclo-)passive with respect to the composed Hamiltonian $\tilde H = H^1(\alpha^1) + H^2(\alpha^2)$:
$$
    \frac{\rm d}{{\rm d}t} \tilde H = \langle \nu^1,y^1\rangle_{\mathcal{U}, \mathcal{Y}} + \langle \nu^2,y^2\rangle_{\mathcal{U}, \mathcal{Y}} \, .
$$
In particular, $H^1$ and $H^2$ only depends on $\alpha^1$ and $\alpha^2$ respectively and no cross-terms are present: the Hamiltonian $\tilde H$ is said to be \emph{separable}.
\subsection{Stokes-Lagrange interconnection}
Let us now take for $i \in \{1,2\}$, two pH systems implicitly defined on a Stokes-Lagrange structure as follows:
$$
    \begin{bmatrix}
        f^i\\e^i
    \end{bmatrix} \in \mathcal{D}^i \subset \mathcal{F}^i  \times \mathcal{E}^i \, , \qquad 
    f^i = - \partial_t \alpha^i \, , \qquad 
    \begin{bmatrix}
        \alpha^i \\ \chi^i \\ e^i \\ \varepsilon^i
    \end{bmatrix} \in \mathcal{L}^i \subset \mathcal{X}^i \times \mathcal{Y} \times \mathcal{E}^i \times \mathcal{U} \, ,
$$
along with the Stokes-Lagrange operators $P^i$, $S^i$, and $\gamma^i$, $\beta^i$. These two systems are defined along with some energy ports which we will use for the interconnection \cite{krhac2024port}.

Denoting by $z^i$ the latent space variable, and following~\eqref{eqn:Hamiltonian-Lagrange-Definition}, their respective Hamiltonian reads either: 
$$
\begin{aligned}
    & H^i_+(z^i) = \frac{1}{2} \langle S^i z^i, P^iz^i \rangle + \frac{1}{2} \langle \beta^i (z^i), \gamma^i (z^i) \rangle \, , \\
    \text{or} \qquad  & 
    H^i_-(z^i) = \frac{1}{2} \langle S^i z^i, P^iz^i \rangle - \frac{1}{2} \langle \beta^i (z^i), \gamma^i (z^i) \rangle \, .
\end{aligned}
$$

\paragraph{Interconnection through the integration of a passive output}
When the plus Hamiltonian $H_+^i$ is chosen, the system is (cyclo-)passive with respect to $( \varepsilon^i,\partial_t \chi^i)$, \textit{i.e.} $\frac{\rm d}{{\rm d}t} H^i_+ = \langle \varepsilon^i, \partial_t \chi^i \rangle_{\mathcal{U}, \mathcal{Y}}$. Following~\cite{borja2023interconnection}, let us denote by $\phi_T: \mathcal{Y} \times \mathcal{Y} \rightarrow \mathbb R$ a differentiable potential, then let us define the following interconnection scheme as:
$$
    \begin{bmatrix}
        \varepsilon^1 \\
        \varepsilon^2
    \end{bmatrix} = - \begin{bmatrix}
        \Id & 0\\
        0 & \Id
    \end{bmatrix} \begin{bmatrix}
        \delta_{\chi^1} \phi_T(\chi^1,\chi^2) \\
        \delta_{\chi^2} \phi_T(\chi^1,\chi^2)
    \end{bmatrix} + \begin{bmatrix}
        \nu^1 \\ \nu^2
    \end{bmatrix}.
$$
\begin{remark}
    Note that in this case, the interconnection is carried on the whole control/interconnection space. One could also consider an interconnection on a subdomain of the boundary (\textit{e.g.} the right boundary of a rectangular plate interconnected with the left boundary of a rectangular plate).
\end{remark}
As a result, the interconnected system becomes passive with respect to the Hamiltonian defined as $\tilde H = H^1_+ + H^2_+ + \phi_T$:
$$
    \begin{aligned}
        \frac{\rm d}{{\rm d}t} \tilde H & = \langle \varepsilon^1, \partial_t \chi^1 \rangle_{\mathcal{U},\mathcal{Y}} + \langle \varepsilon^1, \partial_t \chi^1 \rangle_{\mathcal{U},\mathcal{Y}} \\
        & \qquad + \langle \delta_{\chi^1} \phi_T(\chi^1,\chi^2), \partial_t \chi^1\rangle_{\mathcal{U}, \mathcal{Y}} + \langle \delta_{\chi^2} \phi_T(\chi^1,\chi^2), \partial_t \chi^2 \rangle_{\mathcal{U},\mathcal{Y}} \, , \\
        &= \langle \nu^1, \partial_t \chi^1 \rangle_{\mathcal{U},\mathcal{Y}} + \langle \nu^2, \partial_t \chi^2 \rangle_{\mathcal{U},\mathcal{Y}} \, .
    \end{aligned}
$$
As a consequence, this interconnected system is (cyclo-)passive with respect to an \emph{a priori} non-separable Hamiltonian. In the setting of Control by Interconnection (CBI), \textit{i.e.} when the first system is a physical system and the second a controller, such property can be used to design an interconnection that assigns a desired equilibrium point. However, this method requires the integral of a passive output to be available which is given here by Stokes-Lagrange structure; the passive output being $\partial_t \chi^i$ and its integral $\chi^i$. 

\begin{remark}
    The Stokes-Lagrange structure gives us a direct interpretation of the integral of the passive output as an energy control port. In the case of a mechanical system (\textit{e.g.} Kirchhoff-Love plate model), such an output would be the displacement.
\end{remark}

\begin{example}\label{example:wave-lagrange-interconnection}(Wave equation - continued) In this example we will consider a boundary interconnection using energy port control variables.
    Let us consider two wave equations written in Stokes-Lagrange representation; for $i \in \{1,2\}$, we consider:
    $$
        \frac{\partial}{\partial t} \begin{bmatrix}
            w^i \\ p^i
        \end{bmatrix} = \begin{bmatrix}
            0 & 1 \\ -1 & 0
        \end{bmatrix} \begin{bmatrix}
            N^i \\ v^i
        \end{bmatrix},
    $$
    with the constitutive relations $N^i = - \partial_x E \partial_x w^i$ and $v^i = \frac{1}{\rho}p^i$ and the boundary energy ports:
    $$
        \chi^i = \begin{bmatrix}
            w^i(a) \\ w^i(b)
        \end{bmatrix}, \qquad \varepsilon^i = \begin{bmatrix}
            - E \partial_x w^i(a) \\ E \partial_x w^i(b)
        \end{bmatrix}.
    $$
    Let us now interconnect the wave equation $1$ on the right boundary with the wave equation $2$ on the left boundary thanks to the following potential:
    $$
        V(\chi^1, \chi^2) = V(\chi^1_2,\chi^2_1) = \frac{k}{2}(\chi^1_2 - \chi^2_1)^2 = \frac{k}{2}(w^1(b) - w^2(a))^2 \, ,
    $$
    \begin{figure}[ht]
        \centering
        \begin{tikzpicture}
            \node[draw] (A) at (-4.0,1) {$\chi^1_1,\varepsilon_1^1$};
            \node[draw] (B) at (0,1) {$\chi^1_2,\varepsilon_2^1$};
            \node[draw] (C) at (0.3,-1) {$\chi^2_1,\varepsilon_1^2$};
            \node[draw] (D) at (4.3,-1) {$\chi^2_2,\varepsilon_2^2$};
            \draw (A) -- (B);
            \draw[dashed,thick,,color=red] (B)--(C) node[midway,right] {$V(\chi^1_2,\chi^2_1)  $};
            \draw (C) -- (D);

            \node (E) at (2.3,-1.40) {$w^2,p^2$}; 
            \node (F) at (-2,1.40) {$w^1,p^1$};

            \node (nu12) at (0,2.7) {$\nu^1_2$};
            \node (nu21) at (0.3,-2.5) {$\nu^2_1$};
            \node (nu11) at (-4,2.7) {$\nu^1_1$};
            \node (nu22) at (4.3,-2.5) {$\nu^2_2$};
            \draw[<-,thick,dashed, color=blue] (A) -- (nu11);
            \draw[<-,thick,dashed, color=blue] (D) -- (nu22);

            \draw[<-,thick,dashed, color=blue] (B) -- (nu12);
            \draw[<-,thick,dashed, color=blue] (C) -- (nu21);

            \draw[dotted] (-2.0,1) ellipse (4.25cm and 0.75cm);
            \draw (-4.0,0) node[below left]{Wave eq. 1};
            \draw[dotted] (2.13,-1) ellipse (4.25cm and 0.75cm);
            \draw (4.3,0) node[above right]{Wave eq. 2};
        \end{tikzpicture}
        \caption{Energy port interconnection using a potential $V$.}
        \label{fig:wave-stokes-lagrange-interconnection}
    \end{figure}
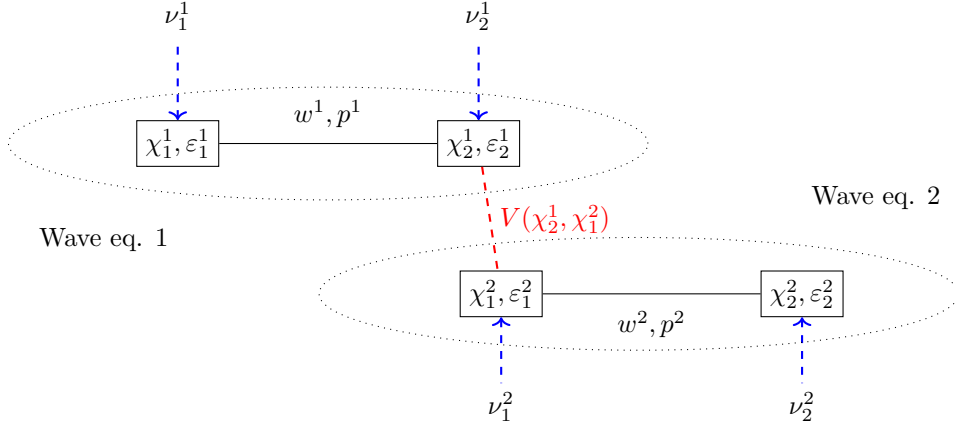
    with $k>0$, see Figure~\ref{fig:wave-stokes-lagrange-interconnection}. This can be interpreted as adding a spring of stiffness $k$ between the two boundaries. This yields the control laws:
        $$
            \begin{bmatrix}
                \varepsilon^1 \\
                \varepsilon^2
            \end{bmatrix} = \begin{bmatrix}
                - \Id & 0\\
                0 & -\Id
            \end{bmatrix} \begin{bmatrix}
                \nabla_{\chi^1} V \\
                \nabla_{\chi^2} V
            \end{bmatrix} + \begin{bmatrix}
                \nu^1 \\ \nu^2
            \end{bmatrix}= k\begin{bmatrix}\begin{bmatrix}
                0\\
                \chi^2_1 - \chi^1_2 
            \end{bmatrix}
                \\ \\
                \begin{bmatrix}
                   \chi^1_2 - \chi^2_1 \\
                0 
                \end{bmatrix}
            \end{bmatrix} + \begin{bmatrix}
                \nu^1 \\ \nu^2
            \end{bmatrix}.
        $$
    Let us now compute the total Hamiltonian $\tilde H = H^1 + H^2 + V$:
    $$
        \tilde H  = \frac{1}{2}\int_a^b \left( E (w^1)^2 + \frac{1}{\rho} (p^1)^2 + E (w^2)^2 + \frac{1}{\rho} (p^2)^2 \right) \, \d x + \frac{k}{2}(w^1(a) - w^2(b))^2 \, ,
    $$
    and the power balance reads:
    $$
        \frac{\rm d}{{\rm d}t} \tilde H = \partial_t \chi^1 \nu^1 + \partial_t \chi^2 \nu^2 \, .
    $$
    In particular, when adding dissipation to the interconnection, \textit{e.g.} $\nu^1_2 = - \gamma \, \partial_t \chi^1_2$ and $\nu^2_1 = - \gamma \, \partial_t \chi^2_1$, with $\gamma > 0$ a scalar, one gets:
    $$
        \frac{\rm d}{{\rm d}t} \tilde H \leq 0 \, .
    $$
\end{example}

\paragraph{Interconnection through differentiation of a passive output} 
Let us choose $H_-^i$ and consider the following nonlinear operator 
$\Psi: \mathcal{X}^1 \times \mathcal{X}^2 \rightarrow \mathcal{L}(\mathcal{Y}\times\mathcal{Y}, \mathbb{R}), (\alpha^1,\alpha^2) \mapsto \psi[\alpha^1,\alpha^2](\cdot,\cdot)$, and assume that it is differentiable with respect to $\alpha^1$, $\alpha^2$, and denote $\beta_i[\alpha^1,\alpha^2]\{u\} := \delta_{\alpha^i} \psi[\alpha^1,\alpha^2]\{u \} \in \mathcal{L}(\mathcal{Y}\times\mathcal{Y}, \mathbb{R}) $. Then, we get:
$$
    \frac{\rm d}{{\rm d}t} \Psi[\alpha^1, \alpha^2] = \beta_1[ \partial_t  \alpha^1, \alpha^2]\{\partial_t \alpha^1\} + \beta_2[\alpha^1, \partial_t  \alpha^2]\{\partial_t \alpha^2\} \in \mathcal{L}(\mathcal{Y}\times\mathcal{Y}, \mathbb{R}) \, .
$$
Let us now define the controls as:
$$
    \varepsilon^1 =  \Psi[\alpha^1,\alpha^2](\cdot,\chi^2) - \int_0^t \left[ \beta_1[  \alpha^1, \alpha^2]\{\partial_t \alpha^1\}( \cdot, \chi^2) \right] \, \d t + \nu^1 \, , \qquad \in \mathcal{Y}' = \mathcal{U} \, ,
$$
$$
    \varepsilon^2 =  \Psi[\alpha^1,\alpha^2](\chi^1,\cdot) - \int_0^t \left[ \beta_2[\alpha^1, \alpha^2]\{\partial_t \alpha^2\}( \chi^1, \cdot)\right]\, \d t  + \nu^2 \, , \qquad \in \mathcal{Y}' = \mathcal{U} \, .
$$
Then, the coupled system is passive with respect to the total Hamiltonian:
$$
    \tilde H = H_-^1 + H_-^2 + \Psi[\alpha^1, \alpha^2](\chi^1,\chi^2) \, ,
$$
and the power balance reads:
$$
    \begin{aligned}
    \frac{\rm d}{{\rm d}t} \tilde H  & = - \langle \chi^1, \partial_t \varepsilon^1\rangle - \langle \chi^2,  \partial_t \varepsilon^2\rangle 
    + \Psi[\alpha^1, \alpha^2](\partial_t \chi^1,\chi^2) + \Psi[\alpha^1, \alpha^2](\chi^1, \partial_t \chi^2) \\
    & \qquad   + \beta_1[   \alpha^1, \alpha^2]\{\partial_t \alpha^1 \}(\chi^1,\chi^2) + \beta_2[\alpha^1,   \alpha^2]\{\partial_t \alpha^2 \}(\chi^1,\chi^2) \, , \\
    & = - \langle \chi^1, \partial_t \nu^1 \rangle - \langle \chi^2, \partial_t \nu^2\rangle \, .
    \end{aligned}
$$
This method uses the fact that $\varepsilon^i$ is differentiated with respect to time in the power balance; hence it allows us to differentiate the passive output $\chi^i$ with respect to time.
\begin{example}\label{example:piezo}(Piezoelectric effect) In this example we will consider a \emph{distributed interconnection} through distributed \emph{energy control ports}. In particular, we will interpret the piezoelectric effect as a coupling through energy ports, yielding the pH representation found in~\cite{Macchelli2004piezo}.
Let us define the 1D domain $\Omega = [a,b]$ and consider two pH systems; the first being the 1-D wave equation and the second the 1-D Maxwell's equations with $\bm{D} = \begin{bmatrix}
    0 & 0 & D
\end{bmatrix}^\top$ and $\bm{B} = \begin{bmatrix}
    0 & B & 0
\end{bmatrix}^\top$:
$$
    \partial_t\begin{bmatrix}
        \varepsilon \\
        p
    \end{bmatrix} = \begin{bmatrix}
        0 & \partial_x \\
        \partial_x &  0
    \end{bmatrix}\begin{bmatrix}
        \sigma \\ v
    \end{bmatrix}, \qquad \partial_t \begin{bmatrix}
        D \\ B
    \end{bmatrix} = \begin{bmatrix}
        0 & \partial_x \\
        \partial_x & 0
    \end{bmatrix}\begin{bmatrix}
        E \\ H
    \end{bmatrix},
$$
with the two \emph{controlled} Hamiltonian defined as:
$$
\begin{aligned}
    H_W(\varepsilon,p,u_1) &= \int_\Omega \frac{1}{2\rho}p^2 + \frac{k}{2} \varepsilon^2 + \varepsilon u_1 \, \d x \, , 
    \\ H_{EM}(D,B,u_2) &= \int_\Omega \frac{1}{2\varepsilon_0} D^2 + \frac{1}{2\mu_0}B^2 + D u_2 \, \d x \, .
\end{aligned}
$$ 
The constitutive relations are given by:
$$
    \sigma = k \varepsilon + u_1 \, , \qquad v = \frac{1}{\rho}p \, , \qquad E = \frac{1}{\varepsilon_0}D + u_2 \, , \qquad H = \frac{1}{\mu_0}B \, .
$$

$u_1$ corresponds to a nonzero stress-free configuration, meaning that the stress-free configuration is attained for a nonzero deformation. $u_2$ corresponds to a dielectric polarization term (usually denoted by $P$~\cite{kovetz2000electromagnetic}). Controlling these terms allows us to control the stress-free configuration and dielectric polarization of the electric field; and we will use these two controls to interconnect the systems through their Hamiltonian functionals.
The two energy ports read:
$$
    u_1, \chi_1 = \frac{\delta H_{W}}{\delta u^1} = \varepsilon \, , \qquad u_2, \chi_2 = \frac{\delta H_{EM}}{\delta u^2} = D \, .
$$
In particular, the latent state for the wave equation is $z = \begin{bmatrix}
    \varepsilon & p & u_1
\end{bmatrix}$, and the Stokes-Lagrange control/observation operators $\gamma$ and $\beta$ select $u_1$ and $\varepsilon$ respectively: $u_1 = \gamma(z), \quad \varepsilon=\beta(z)$. Moreover, writing the constitutive relations with the energy ports reveals the symmetry:
$$ \begin{bmatrix}
    \sigma \\ v \\ \chi_1
\end{bmatrix} = \begin{bmatrix}
    k & 0 & 1\\
    0 & \frac{1}{\rho} & 0\\
    1 & 0 & 0
\end{bmatrix} \begin{bmatrix}
    \varepsilon \\ p \\ u_1
\end{bmatrix}, \qquad \begin{bmatrix}
    E \\ H \\ \chi_2
\end{bmatrix} = \begin{bmatrix}
    \frac{1}{\varepsilon_0} & 0 & 1\\
    0 & \frac{1}{\mu_0} & 0\\
    1 & 0 & 0
\end{bmatrix} \begin{bmatrix}
    \varepsilon \\ p \\ u_2
\end{bmatrix}. $$

Let us now define $q \in L^\infty(\Omega), \, q(x)\geq0$ a distributed parameter and $\alpha_1 = \begin{bmatrix}
    \varepsilon & p
\end{bmatrix}^\top$, $\alpha_2 = \begin{bmatrix}
    D & B
\end{bmatrix}^\top$ the state variables. Then, let us interconnect the two systems by considering the potential:
$$
    \Psi(\alpha_1,\alpha_2)[\chi_1,\chi_2] := - \langle q \chi_1,\chi_2 \rangle_{L^2} =  - \int_\Omega q \,  \chi_1 \chi_2 \, \d x \, .
$$
In particular, $\beta_1 = \frac{\delta \Psi}{\delta \alpha_1} = 0$, $\beta_2 = \frac{\delta \Psi}{\delta \alpha_2} = 0$.
This gives us the following controls:
$$
    u_1 = -\Psi(\alpha_1,\alpha_2)[\cdot, \chi_2] =  q \, \chi_2 =  q \, D \, , \quad u_2 = - \Psi(\alpha_1,\alpha_2)[\chi_1, \cdot] = q \, \chi_1 =  q \, \varepsilon \, ,
$$
where we used the fact that $\Psi(\alpha_1,\alpha_2)$ is a continuous bilinear function on a Hilbert space to identify $u_1$ and $u_2$ as vectors.
Computing the total Hamiltonian $\tilde H = H_{W} + H_{EM} + \Psi(\alpha_1,\alpha_2)[\chi_1,\chi_2]$ yields:
$$
    \tilde H(\varepsilon,p,D,B) = \frac{1}{2}\int_\Omega \frac{1}{\rho}p + k \varepsilon  + \frac{1}{\varepsilon_0} D^2 + \frac{1}{\mu_0}B^2 + 2q\, D \varepsilon \, \d x \, .
$$
Computing the variational derivatives with respect to $\varepsilon$ and $D$ gives us:
$$
    \sigma = \frac{\delta \tilde H}{\delta \varepsilon} = k \varepsilon + q D \, , \qquad E = \frac{\delta \tilde H}{\delta D} = \frac{1}{\varepsilon_0} D + q \varepsilon \, ,
$$
which corresponds to the model presented in~\cite{macchelli2011energy}.
\end{example}

\begin{remark}
    Both methods are an extension of CBI, in particular and contrarily to CBI, they do not, however, suffer from the Dissipation Obstacle that occurs when using power boundary ports as was shown in~\cite{borja2023interconnection}. The Dissipation Obstacle prevents the stabilization of the system in the subspace where dissipation occurs~\cite{ortega2008control}, \textit{e.g.} non-zero velocity for mechanical systems with dissipation.
\end{remark}

\section{Conclusion \& Outlook}

We presented a definition of the Stokes-Lagrange structure extending the classical definition of a Hamiltonian functional;  this allows for describing classical pH systems in various representation and defining implicit constitutive relations which are necessary when studying \textit{e.g.} \emph{nonlocal} phenomenons. Through a set of examples, namely the Reissner-Mindlin and Kirchhoff-Love plate models, Maxwell's equation of electromagnetism and Dzektser's equation; we showed how these different representations and implicit constitutive relations are obtained. Moreover, we showed how this structure enables the use of energy ports as controls and gives a physical interpretation (\textit{e.g.} boundary displacement) based on the system's operators; which yields a new set of tools to design passivity-based controllers.

Moreover, it is yet to describe the structure-preserving discretization of Stokes-Lagrange structures (see~\cite{brugnoli2019portI} for the Stokes-Dirac case) though some examples have already been treated in $1$D (see~\cite{bendimerad2023implicit}) and an extension to $N$-dimensional case has to be carried out.
These representations coupled with a structure-preserving discretization could be useful when studying crack propagation where nonlocal relations are present; moreover, they will allow the study of nonlocal boundary conditions for the Maxwell's equation in the pH formalism.

\bibliographystyle{siam} 
\bibliography{biblio}

\appendix

\section{Operator transposition} \label{apx:operator-transposition} 


In this appendix, we will show that the Stokes-Dirac and Stokes-Lagrange representations of the wave-like system \eqref{eqn:sys-K-matrix} are properly defining a Stokes-Dirac and Stokes-Lagrange structures respectively. 

To do so, we will assume the two following hypothesis which follows from Assumption~\ref{eqn:skew-symmetry-assumption} on the operators $\bm{K}, \gamma, \beta$ that allowed us to construct the Stokes-Dirac and Stokes-Lagrange representation and will be used to prove that these are defining Stokes-Dirac and Stokes-Lagrange structures. These are: \begin{itemize}
    \item  $\gamma$ and $\beta$ have dense image in $\mathcal{Y}$ and $\mathcal{U}$ respectively.
    \item $\ker(\gamma)$ and $\ker(\beta)$ are dense  in $L^2(\Omega,\mathbb{R}^n)$ and $L^2(\Omega,\mathbb{R}^m)$ respectively.
    \end{itemize}

\subsection{Stokes-Dirac representation}
Let us consider the Stokes-Dirac representation \eqref{eqn:stokes-dirac-sys-k}:
$$
        \begin{bmatrix}
            \partial_t \alpha^{SD}_1 \\
            \partial_t \alpha^{SD}_2
        \end{bmatrix} = \underbrace{\begin{bmatrix}
            0 &  \bm{K}  \\ -\bm{K}^\dag & 0
        \end{bmatrix}}_{=: J^{SD}} \begin{bmatrix}
            e_1^{SD} \\ e_2^{SD}
        \end{bmatrix}, \qquad
    \begin{cases}
        e^{SD}_1 = \bm{\eta} \, \alpha_1^{SD}, \\
        e^{SD}_2 = \bm{\kappa} \, \alpha_2^{SD},
    \end{cases}
$$
and show that this system is defined on a Stokes-Dirac structure. 

Firstly, we will define the flows and effort spaces as $\mathcal{F}_s = \mathcal{E}_s = L^2(\Omega, \mathbb{R}^m)\times L^2(\Omega, \mathbb{R}^n)$, and $\mathcal{F}_u = \mathcal{Y}$, $\mathcal{E}_u = \mathcal{U}$.

Secondly, let us consider the Stokes-Dirac structure operator $J^{SD}$ along with the control and observation operators $K = \begin{bmatrix}
    0 & \gamma
\end{bmatrix}$ and  $G = \begin{bmatrix}
    \beta & 0
\end{bmatrix}$.

Thirdly, let us define the appropriate spaces: $\mathcal{W}_1 = \mathcal{D}(J^{SD}) \subset \mathcal{E}_s$ and $\mathcal{W}_0 = \ker(G) \cap \ker(K) \subset \mathcal{D}(J^{SD})$. Note that one can deduce that $\mathcal{W}_0 = \ker(\beta) \times \ker(\gamma)$ which is dense in $L^2(\Omega,\mathbb{R}^m) \times L^2(\Omega,\mathbb{R}^n) =  \mathcal{E}_s$ by assumption.

Let us now show that $J^{SD}$, $K$ and $G$ satisfy the three conditions of Assumption~\ref{eqn:skew-symmetry-assumption}:

\begin{enumerate}
    \item Let $x=\begin{bmatrix}
        x_1 \\ x_2
    \end{bmatrix} \in \mathcal{W}_1$, then computing $\left \langle \begin{bmatrix}
                x_1 \\ x_2
            \end{bmatrix}, J^{SD} \begin{bmatrix}
                x_1 \\ x_2
            \end{bmatrix} \right \rangle_{\mathcal{E}_s}$ yields: \begin{equation*}
        \begin{aligned}
            \left \langle \begin{bmatrix}
                x_1 \\ x_2
            \end{bmatrix}, J^{SD} \begin{bmatrix}
                x_1 \\ x_2
            \end{bmatrix} \right \rangle_{\mathcal{E}_s} &= \langle x_1, \bm{K} x_2 \rangle_{L^2} - \langle x_2, \bm{K}^\dag x_1 \rangle_{L^2} \\
            & = \langle \begin{bmatrix}
                \beta & 0
            \end{bmatrix}x, \begin{bmatrix}
                0 & \gamma 
            \end{bmatrix} x \rangle_{\mathcal{E}_u,\mathcal{F}_Y} \\
            &= \langle Gx,Kx \rangle_{\mathcal{E}_u, \mathcal{F}_u}.
        \end{aligned}
    \end{equation*}
    \item $J_0 := J_{|\mathcal{W}_0}$ is densely defined as $\mathcal{W}_0$ is dense in $\mathcal{E}_s$. Let us now prove that  satisfies $J_0^* = - J$. From the block structure of $J$, this gives us the following conditions:
$$
     (\bm{K}^\dag_{|\ker(\beta)})^* = \bm{K} \, , \qquad (\bm{K}_{|\ker(\gamma)})^* = \bm{K}^\dag \, , 
$$
which is true by the definition of the formal adjoint. 

    \item $\gamma$ and $\beta$ have dense image in $\mathcal{Y}$ and $\mathcal{U}$ respectively, hence, $\begin{bmatrix}
        K \\ G
    \end{bmatrix} = \begin{bmatrix}
        \begin{bmatrix}
            0 & \gamma
        \end{bmatrix} \\
        \begin{bmatrix}
            \beta & 0
        \end{bmatrix}
    \end{bmatrix}$ has a dense range in $\mathcal{F}_u\times \mathcal{E}_u$.
\end{enumerate}

This shows that \eqref{eqn:stokes-dirac-sys-k} is defined on a Stokes-Dirac structure.

\subsection{Stokes-Lagrange representation}

Let us now study the Stokes-Lagrange representation \eqref{eqn:stokes-lagrange-sys-k}. To do so, let us define the Lagrange structure operators $P$, $S$ and $\tilde \gamma$ and $\tilde \beta$ the observation and control operators as:
$$
P = \begin{bmatrix}
    \Id & 0 \\
    0 & \Id
\end{bmatrix}, \quad S = \begin{bmatrix}
    \bm{K}^\dag \bm{\eta} \bm{K} & 0\\
    0 & \bm{\kappa}
\end{bmatrix}, \, \text{ and } \, \tilde \gamma = \begin{bmatrix}
    - \gamma \\ 0
\end{bmatrix}, \quad \tilde \beta = \begin{bmatrix}
     \beta(\bm{\eta} \bm{K} \cdot ) \\ 0
\end{bmatrix},
$$
along with the appropriate spaces $\mathcal{Z} = \mathcal{X}=\mathcal{E}=L^2(\Omega,\mathbb{R}^n)\times L^2(\Omega,\mathbb{R}^n)$,  $\mathcal{Z}_1 = \mathcal{D}(\bm{K}^\dag \bm{\eta} \bm{K}) \times  L^2(\Omega,\mathbb{R}^n)$ and $\mathcal{Z}_0 = \ker(\tilde\gamma)\cap\ker(\tilde\beta)$. To prove that these operators define a Stokes-Lagrange subspace, let us assume that
\begin{itemize}
    \item $\tilde \beta$ has a dense image in $\tilde{\mathcal{U}} := \mathcal{U}\times \{\bm 0\}$ (the density of the image of $\tilde \gamma$ in $\tilde{\mathcal{U}} := \mathcal{Y}\times \{\bm 0\}$ is immediate by assumption);
    \item $\mathcal{Z}_0$ is dense in $\mathcal{Z}$.
\end{itemize}
\begin{remark}
    Note that cases exist where $\beta$ has a dense image but $\tilde \beta$ does not; in particular, if $\bm \eta$ is null on a part of the boundary. In such cases, one can restrict the control space from $\mathcal{U}$ to the subspace $\tilde{\mathcal{U}} = \overline{\Ima(\beta \bm \eta \bm K)} \times \{\bm 0\} $.
\end{remark}
Now, let us prove that these operators satisfy Assumptions~\ref{eqn:symmetry-assumption},~\ref{eqn:maximality-assumption},~\ref{eqn:surjectivity-assumption}, hence that they define a Stokes-Lagrange structure. 

\begin{lemma}(Unbounded operator) \label{lemma:boundedness-transposition} $P$ and $S$ satisfy the boundedness assumption~\ref{eqn:bounded-assumption}.    
\end{lemma}

\begin{lemma}(Symmetry) \label{lemma:symmetry-transposition} The quadruple $(P,\tilde \gamma, S, \tilde \beta)$ satisfies the symmetry assumption~\ref{eqn:symmetry-assumption}; for all $z_1, z_2 \in \mathcal{Z}_1$:
    $$
        \langle \tilde \beta z_2, \tilde \gamma z_1 \rangle_{\mathcal{U},\mathcal{Y}} + \langle Sz_2, Pz_1 \rangle_{\mathcal{E},\mathcal{X}} = \langle Sz_1, Pz_2 \rangle_{\mathcal{E},\mathcal{X}} + \langle \tilde \beta z_1, \tilde \gamma z_2 \rangle_{\mathcal{U},\mathcal{Y}} \, .
    $$
\end{lemma}

\begin{lemma}(Maximality) \label{lemma:maximality-transposition} The couple $(P,S)$ is maximally reciprocal with respect to $\mathcal{Z}_0$.
\end{lemma}
\begin{proof}
Let  $x,e \in \mathcal{X}\times\mathcal{E}$ such that:
$$ 
\forall z_0 \in \mathcal{Z}_0, \quad \langle S z_0, x\rangle_{\mathcal{E},\mathcal{X}} - \langle e, P z_0\rangle_{\mathcal{E},\mathcal{X}}  = 0 \, , 
$$
and prove that there exists $Z \in \mathcal{Z}_1$ such that $x = Pz$ and $e = Sz$. Let us define $\tilde{z} = x$, which trivially yields: $ x = P \tilde z = \tilde z \in \mathcal{Z}$. 

Let us now show that 
$\tilde z\in \mathcal{Z}_1$ and $e = S \tilde z$. Firstly, we have that:
$$ 
\forall z \in \mathcal{Z}_0, \quad \langle S z,  \tilde z \rangle_{\mathcal{E},\mathcal{X}} - \langle e, z\rangle_{\mathcal{E},\mathcal{X}} = \langle S z, P \tilde z \rangle_{\mathcal{E},\mathcal{X}} - \langle e, P z\rangle_{\mathcal{E},\mathcal{X}}  = 0 \, , 
$$
in particular, defining $S_0 := S_{|\mathcal{Z}_0}$, we have:
$$
\forall z_0 \in \mathcal{Z}_0, \quad \langle e, z_0\rangle_{\mathcal{E},\mathcal{X}} = \langle S_0 z_0, \tilde z \rangle_{\mathcal{E},\mathcal{X}} \, ,
$$
hence, $\tilde z \in \mathcal{D}(S_0^*)$. Now using the block structure of $S$ leads to:
$$
\forall \begin{bmatrix}
    z_0^1 \\ z_0^2
\end{bmatrix} \in \mathcal{Z}_0, \qquad \begin{cases}
    \langle e^1, z_0^1 \rangle_{\mathcal{E},\mathcal{X}} = \langle \bm{K}^\dag(\bm{\eta} \bm{K}) z_0^1 , \tilde{z}^1\rangle_{\mathcal{E},\mathcal{X}} \, , \\
    \langle e^2, z_0^2 \rangle_{\mathcal{E},\mathcal{X}} = \langle \bm{\kappa} z_0^2,  \tilde{z}^2\rangle_{\mathcal{E},\mathcal{X}} =  \langle  z_0^2,  \bm{\kappa} \tilde{z}^2\rangle_{\mathcal{E},\mathcal{X}} \, .
\end{cases}
$$
Since $\mathcal{X}$ is a Hilbert space, we can identify $\mathcal{X}$ with its dual and the second line gives us that $e^2 = \bm{\kappa} \tilde{z}^2$.

Let us focus on the second line. We have: 
$$
\langle e^1, z_0^1 \rangle_{\mathcal{E},\mathcal{X}} = \langle \bm{K}^\dag(\bm{\eta} \bm{K}) z_0^1 , \tilde{z}^1\rangle_{\mathcal{E},\mathcal{X}} \, ,
$$
with $z_0^1 \in \ker(\gamma) \cap \ker(\beta \bm \eta \bm K)$. Hence, given that $\bm \eta \bm K z_0^1 \in \ker(\beta)$ and that $(\bm K^\dag_{|\ker(\beta)})^* = \bm K$, this yields:
$$ 
\langle \bm{K}^\dag(\bm{\eta} \bm{K}) z_0^1 , \tilde{z}^1\rangle_{\mathcal{E},\mathcal{X}} =  \langle \bm{\eta} \bm{K} z_0^1 , \bm{K} \tilde{z}^1\rangle_{\mathcal{E},\mathcal{X}} = \langle  \bm{K} z_0^1 , \bm{\eta} \bm{K} \tilde{z}^1\rangle_{\mathcal{E},\mathcal{X}} \, . 
$$
Now using the fact that $z_0^1 \in \ker(\gamma)$, we get:
$$
\langle  \bm{K} z_0^1 , \bm{\eta} \bm{K} \tilde{z}^1\rangle_{\mathcal{E},\mathcal{X}} = \langle   z_0^1 , \bm{K}^\dag(\bm{\eta} \bm{K} )\tilde{z}^1\rangle_{\mathcal{E},\mathcal{X}} \, .
$$
Finally, this yields, since $\mathcal{X}=\mathcal{E}$:
$$
     \langle e^1, z_0^1 \rangle_{\mathcal{X}} = \langle   z_0^1 , \bm{K}^\dag(\bm{\eta} \bm{K} )\tilde{z}^1\rangle_{\mathcal{X}} \, ,
$$
which, by density of $\mathcal{Z}_0$, gives us that $\tilde{z}^1 \in \mathcal{D}(\bm{K}^\dag(\bm{\eta} \bm{K}))$ and $e^1 = \bm{K}^\dag(\bm{\eta} \bm{K} )\tilde{z}^1$.

Gathering the results yields $\tilde z \in \mathcal{Z}_1$ and $x = P \tilde z$, $e = S \tilde z$; which proves the maximality.
\end{proof}

\begin{lemma}(surjectivity) \label{lemma:surjectivity-transposition} The operator $ \begin{bmatrix}
    \tilde \gamma  \\ \tilde \beta 
\end{bmatrix}$ has a dense image in $\tilde{\mathcal{Y}} \times \tilde{\mathcal{U}}$.
\end{lemma}

Now applying Theorem~\ref{thm:lagrange-iso-coisotropic} with lemmas~\ref{lemma:boundedness-transposition},~\ref{lemma:symmetry-transposition},~\ref{lemma:maximality-transposition} and~\ref{lemma:surjectivity-transposition}, shows that $P,S,\tilde \gamma, \tilde \beta$ defines a Stokes-Lagrange structure.

\subsection{Proof of Theorem~\ref{thm:dirac-lagrange-equiv}} \label{apx:operator-transposition-representation}
\begin{proof}
Let us compute the three different lines directly as follows:
    \begin{itemize}
        \item $\alpha^{SD} = G \alpha^{SL}$: 
        $$
            \begin{aligned}
            G \alpha^{SL}  = \begin{bmatrix}
                \bm{K} & 0\\
                0 & \Id
            \end{bmatrix} \begin{bmatrix}
                \alpha^{SL}_1 \\ \alpha^{SL}_2
            \end{bmatrix}
           = \begin{bmatrix}
            \bm{K}\alpha^{SL}_1  \\ \alpha^{SL}_2
        \end{bmatrix} = \begin{bmatrix}
            \alpha^{SD}_1 \\ \alpha^{SD}_2
        \end{bmatrix}. \end{aligned}
        $$
        \item $G^\dag e^{SD} = e^{SL}$: 
        $$
            G^\dag e^{SD} = \begin{bmatrix}
            \bm{K}^\dag & 0 \\
            0 & \Id
        \end{bmatrix} \begin{bmatrix}
            e^{SD}_1  \\ e^{SD}_2
        \end{bmatrix} = \begin{bmatrix}
            \bm{K}^\dag \bm{\eta} \, \alpha_1^{SD} \\
            \bm{\kappa} \, \alpha_2^{SD}
        \end{bmatrix} = \begin{bmatrix}
            \bm{K}^\dag  \bm{\eta} \bm{K} \,  \alpha_1^{SL} \\
            \bm{\kappa} \, \alpha_2^{SL}
        \end{bmatrix} = \begin{bmatrix}
            e_1^{SL} \\ e_2^{SL}
        \end{bmatrix}.
        $$
        \item $GJ^{SL}G^\dag = J^{SD}$:
        $$
            GJ^{SL}G^\dag = \begin{bmatrix}
            \bm{K} & 0\\
            0 & \Id
        \end{bmatrix} \begin{bmatrix}
            0 & \Id \\ - \Id & 0
        \end{bmatrix} \begin{bmatrix}
            \bm{K}^\dag & 0\\
            0 & \Id
        \end{bmatrix} = \begin{bmatrix}
           0 & \bm{K} \\
           - \bm{K}^\dag & 0
        \end{bmatrix} = J^{SD}.
        $$
    \end{itemize}
    
\end{proof}

\section{Maxwell reciprocity conditions} \label{apx:maxwell-reciprocity}

	The goal of this appendix is to prove that the Maxwell reciprocity conditions are sufficient in order to define a Hamiltonian given a set of implicit constitutive relations.
	Let $\mathcal{X}$ be a Hilbert space, and consider $\mathcal{E} = \mathcal{X}'$. Moreover, let us consider $\mathcal{Z}$ the latent state space and assume that there exists (possibly nonlinear) operators $p: \mathcal{Z} \rightarrow \mathcal{X}$ and $s: \mathcal{Z} \rightarrow \mathcal{E}$ such that:
\begin{equation} \label{eqn:cons-rel-impl}
\mathcal{L} = \begin{bmatrix}
		p \\ s
	\end{bmatrix} \mathcal{Z} \, ,
\end{equation}
    \textit{i.e.} $\alpha = p(z)$ and $e = s(z)$. Moreover, we require that the first variational derivative of $s$ and $p$ exist and the second variational derivative of $p$ exists. These will be denoted by $\delta_z s [z](\cdot), \, \delta_z p [z](\cdot), \, \delta_z^2 p [z](\cdot,\cdot)$.
	Now, let us search for  a Hamiltonian corresponding to \eqref{eqn:cons-rel-impl}. Following ~\cite{hairer2006geometric}, given time-dependent latent state, state and effort $z(t,x), \, \alpha(t,x), \, e(t,x)$, such a Hamiltonian should satisfy the power balance:
	   \begin{equation}\label{eqn:ham-time-deriv-condition}
	       \frac{\rm d}{{\rm d}t} H(z) = \langle e, \partial_t \alpha \rangle_{\mathcal{E},\mathcal{X}} \, .
	   \end{equation} 
    This relation means that the power balance exists and is exactly equal to the scalar product between the efforts $e$ and flows $\partial_t \alpha$. Expanding the right-hand side yields:
	\begin{equation}\label{eqn:ham-time-deriv-condition-expanded}
 \frac{\rm d}{{\rm d}t} H(z) = \langle s(z) \, , \, \delta_z p[z](\partial_t z) \rangle \, .
 \end{equation}
    Let us now define a Hamiltonian candidate given $p$ and $s$ by using \eqref{eqn:ham-time-deriv-condition}. To do so, given a state $z(x)$, define $\overline{z}(t,x) = t z(x)$, a segment between the state $\bm 0 = \overline{z}(0,\cdot)$ and $z = \overline{z}(1,\cdot)$, see Figure~\ref{fig:def-segment}. Then integrating the power balance \eqref{eqn:ham-time-deriv-condition-expanded} along the interval $[0,1]$ gives us:
	\begin{equation} \label{eqn:ham-candidate-def}
		H(z) - H(\bm 0) = \int_0^1 \frac{\rm d}{{\rm d}\tau} H(\overline{z}(\tau,x))\, \d \tau = \int_0^1 \langle s(\tau  z), \delta_z p[\tau  z] z \rangle \d \tau \, .
	\end{equation}
    Note that in this case $H$ is defined up to a constant $H(\bm 0)$ which can be taken arbitrarily (\textit{e.g.} $H(\bm 0) = 0$ or $H(\bm 0) =c$ such that $H(z) \geq 0$ for all $z \in \mathcal{Z}$).
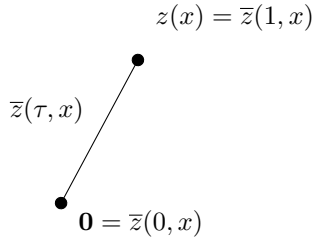
\begin{figure}[ht]
	\centering
	\begin{tikzpicture}
	\node (Z) at (4.25,0.75) {$z(x) = \overline{z}(1,x)$};
	
	\node (O) at (3,-2) {$\bm 0 = \overline{z}(0,x)$};
	
	\draw[*-*] (O)+(-1.1,0.2) -- (3,.25);
	\node (ZZ) at (1.75,-0.5) {$\overline{z}(\tau,x)$};	
    
	\end{tikzpicture}
	\caption{Definition of the Hamiltonian candidate.}
    \label{fig:def-segment}
\end{figure}
Let us now show that the Maxwell reciprocity conditions ensures that this Hamiltonian is well-defined and satisfies \eqref{eqn:ham-time-deriv-condition}. These reciprocity conditions can be stated as the following symmetry assumption:

 \begin{definition} \label{def:maxwell-reciprocity-condition}(Maxwell reciprocity condition) Let $\mathcal{Z,X,E}$ be three Hilbert spaces: $\mathcal{Z}$ the latent state space, $\mathcal{X}$ the state space and $\mathcal{E}$ the effort space. Let  $p:\mathcal{Z}\rightarrow \mathcal{X}$ and $s:\mathcal{Z}\rightarrow\mathcal{E}$ be two operators, such that the first variational derivative of $p$ and $s$ exists. Then, $p$ and $s$ satisfy the Maxwell reciprocity conditions if:
     $$
     \forall z\in \mathcal{Z}, \qquad  \delta_zs^*(z) \, \delta_zp(z) = \delta_zp^*(z) \, \delta_zs(z) \, .
     $$
 \end{definition}
 Then, given $p$ and $s$ satisfying the Maxwell reciprocity conditions and time-dependent latent state $z(t,x)$, let us expand $\frac{\rm d}{{\rm d}t}H(z)$:
	$$
		\begin{array}{rl}
			\dsp \frac{\rm d}{{\rm d}t}H(z) =& \dsp \int_0^1 \langle \delta_z s [\tau z] \tau \partial_t z, \delta_z p[\tau  z] \, z \rangle \\
			& \dsp \qquad + \langle s(\tau  z),  \delta_z p[\tau z](\partial_t z) + \delta_{z^2}^2p[\tau z](\tau \partial_t z,z) \rangle \, \d \tau \, , \\
			=& \dsp \int_0^1 \langle \delta_z p [\tau z] \tau \partial_t z, \delta_z s[\tau  z] \, z \rangle \qquad \qquad \qquad \quad \text{(symmetry)} \\
			& \dsp \qquad \quad + \langle s(\tau  z),  \delta_z p[\tau z](\partial_t z) + \delta_{z^2}^2p[\tau z](\tau \partial_t z,z) \rangle \, \d \tau \, , \\
			=& \dsp \int_0^1 \langle s(\tau z) + \delta_z s [\tau z] \, \tau z, \delta_zp[\tau z] \partial_t z\rangle \\
			& \dsp \qquad \quad + \left\langle \tau s(\tau z), \frac{\rm d}{{\rm d} \tau} \left[ \delta_{z}p[\tau z] \partial_t z \right] \right\rangle \, \d \tau \, , \\
			=& \dsp \int_0^1 \left\langle \frac{\rm d}{{\rm d} \tau}\left[ \tau s(\tau z) \right], \delta_zp[\tau z] \partial_t z \right\rangle + \left\langle \tau s(\tau z), \frac{\rm d}{{\rm d} \tau} \left[ \delta_{z}p[\tau z] \partial_t z \right] \right\rangle \, \d \tau \, , \\
			=& \dsp \int_0^1 \frac{\rm d}{{\rm d}\tau} \langle \tau s(\tau z), \delta_z s[\tau z] \partial_t z \rangle \d \tau \, , \\
			=& \dsp \langle s(z), \delta_zp[z] \partial_t z \rangle \, ,
		\end{array}
	$$
which proves that $H$ satisfies the power balance \eqref{eqn:ham-time-deriv-condition}. Hence, the  definition of a Hamiltonian given a set of nonlinear operators $p$, $s$ \eqref{eqn:ham-candidate-def} is possible if they satisfy the Maxwell's reciprocity condition:
\begin{theorem}
    Let $\mathcal{Z,X,E}$ be three Hilbert spaces: $\mathcal{Z}$ the latent state space, $\mathcal{X}$ the state space and $\mathcal{E}$ the effort space. Let  $p:\mathcal{Z}\rightarrow \mathcal{X}$ and $s:\mathcal{Z}\rightarrow\mathcal{E}$ be two operators, such that the first variational derivative of $p$ and $s$ exists and the second variational derivative of $p$ exists.

    Then, if $p$ and $s$ satisfy the Maxwell reciprocity condition, then there exists $H: \mathcal{Z} \rightarrow \mathbb{R}$ such that:
    $$
        \frac{\rm d}{{\rm d}t} H(z) = \langle e, \partial_t \alpha \rangle_{\mathcal{E},\mathcal{X}} \, ,
    $$
    with $e = s(z)$ and $\alpha = p(z)$.
\end{theorem}

\section{Stokes-Lagrange structure} \label{apx:stokes-lagrange}

In this appendix, we will show additional results and characterizations of the Stokes-Lagrange structure. Firstly, Assumption~\ref{eqn:maximality-assumption} is not the only good definition of maximality; the following theorem allows for a different one.
\begin{lemma} \label{lemma:maximality-equivalence}
    Defining $S_0 := S_{|\mathcal{Z}_0}$, $P_0:= S_{|\mathcal{Z}_0},$ Assumption~\ref{eqn:maximality-assumption} is equivalent to:
    $$ \label{eqn:lagrange-maximally-reciprocal-include}
        \ker \begin{bmatrix}
        S_0^* & 
        -P_0^*
    \end{bmatrix} \subset \Ima \begin{bmatrix}
        P \\ S
    \end{bmatrix}.
    $$
\end{lemma}
\begin{proof} 
\begin{itemize}
    \item[$\implies$]Let us assume that Assumption~\ref{eqn:maximality-assumption} is satisfied and show that $\ker \begin{bmatrix}
        S_0^* & 
        -P_0^*
    \end{bmatrix} \subset \Ima \begin{bmatrix}
        P \\ S
    \end{bmatrix}.
    $

Let $(x,e)^\top \in \ker \begin{bmatrix}
        S_0^* & 
        -P_0^*
    \end{bmatrix}$. Then,
    for all $ (z, z_0) \in \mathcal{Z}_1 \times \mathcal{Z}_0$, $\langle S z, x \rangle_{\mathcal{E},\mathcal{X}} - \langle e ,P z \rangle_{\mathcal{E},\mathcal{X}} = 0$. By applying Assumption~\ref{eqn:maximality-assumption}, there exists $\tilde z$ such that $S\tilde z = e$ and $P \tilde z = x$, hence $(x,e)^\top \in \Ima \begin{bmatrix}
        P \\ S
    \end{bmatrix}$. This being true for all $(x,e)^\top \in \ker \begin{bmatrix}
        S_0^* & -P_0^*
    \end{bmatrix}$, we deduce that:
$$
        \ker \begin{bmatrix}
        S_0^* & 
        -P_0^*
    \end{bmatrix} \subset \Ima \begin{bmatrix}
        P \\ S
    \end{bmatrix}.
    $$

    \item[$\impliedby$] Let us assume that $\ker \begin{bmatrix}
        S_0^* & 
        -P_0^*
    \end{bmatrix} \subset \Ima \begin{bmatrix}
        P \\ S
    \end{bmatrix}$. 
    
    Let $(x,e)^\top\in \mathcal{X}\times\mathcal{E}$ such that for all $ z_0 \in  \mathcal{Z}_0$, $\langle S z_0, x \rangle_{\mathcal{E},\mathcal{X}} - \langle e ,P z_0 \rangle_{\mathcal{E},\mathcal{X}} = 0$ and show that there exists $z \in \mathcal{Z}_1$ such that: $x = Pz$ and $e=Sz$. 
    
     Since $\langle S z_0, x \rangle_{\mathcal{E},\mathcal{X}} - \langle e ,P z_0 \rangle_{\mathcal{E},\mathcal{X}} = 0$ is true for all $z\in \mathcal{Z}_0$ , we have that $(x,e) \in \ker \begin{bmatrix}
        S_0^* & 
        -P_0^*
    \end{bmatrix}$. Moreover, by assumption, we have  that:  $ \ker \begin{bmatrix}
        S_0^* & 
        -P_0^*
    \end{bmatrix} \subset \Ima \begin{bmatrix}
        P \\ S
    \end{bmatrix}$. This implies that there exists $z \in \mathcal{Z}_1$ such that: $x = Pz$ and $e=Sz$. 
    
    To conclude, we have that:

    If $(x,e)^\top \in \mathcal{X}\times\mathcal{E}$ is such that for all $ z_0 \in  \mathcal{Z}_0$, $\langle S z_0, x \rangle_{\mathcal{E},\mathcal{X}} - \langle e ,P z_0 \rangle_{\mathcal{E},\mathcal{X}} = 0$ then, there exists $z \in \mathcal{Z}_1$ such that: $x = Pz$ and $e=Sz$. Which is exactly Assumption~\ref{eqn:maximality-assumption}.
\end{itemize}
\end{proof}

\begin{remark}
    Definition~\ref{def:Stokes-Lagrange-candidate} uses the image representation of a Lagrange structure. However, Theorem~\ref{thm:lagrange-iso-coisotropic} shows that one can consider a so-called kernel representation:
    $$
        \mathcal{L} 
        = \left  \lbrace \begin{bmatrix}
            x \\ \chi \\ e \\ \varepsilon
        \end{bmatrix} \in  \mathcal{X} \times \mathcal{Y} \times \mathcal{E} \times \mathcal{U} \, \mid \, \forall z \in \mathcal{Z}_1, \, 
        \left\langle \left\langle \begin{bmatrix}
                x \\ \chi \\ e \\ \varepsilon
            \end{bmatrix}, \begin{bmatrix}
                P \\ \gamma \\ S \\ \beta
            \end{bmatrix}z \right\rangle \right\rangle_- = 0 \, \right\rbrace.
    $$
    This is the usual representation in the explicit finite-dimensional case, with $ S^\top x = P^\top e $. In particular, if $P = I_n$, $S = Q = Q^\top \geq0$, this becomes $e = Qx$. Finally, defining $H(x) = \frac{1}{2} x^\top Q x$ we get the usual $e = \nabla H(x) = Q x$.
\end{remark}

\subsection{Proof of Theorem~\ref{thm:lagrange-iso-coisotropic}} \label{apx:stokes-lagrange-proof}
Let us demonstrate Theorem~\ref{thm:lagrange-iso-coisotropic}, \textit{i.e.} that a Stokes-Lagrange structure is a Lagrange structure:
\begin{proof}
    $\subset$: Applying Assumption~\ref{eqn:symmetry-assumption} shows that $\mathcal{L} \subset \mathcal{L}^{\perpext}$ directly.

    $\supset$: Let $ \begin{bmatrix}
        x & \chi & e & \varepsilon
    \end{bmatrix}^\top \in \mathcal{L}^\perpext$, then for all $z \in \mathcal{Z}_1,$ we have:
        \begin{equation} \label{eqn:orthogonal-relation}
            0 = \left\langle \left\langle \begin{bmatrix}
                x \\ \chi \\ e \\ \varepsilon
            \end{bmatrix}, \begin{bmatrix}
                P \\ \gamma \\ S \\ \beta
            \end{bmatrix}z \right\rangle \right\rangle_- .
        \end{equation}
    In particular, by choosing $z\in  \mathcal{Z}_0$, we have: $ 0 = \langle e, Pz \rangle_{\mathcal{E},\mathcal{X}} - \langle Sz, \rangle_{\mathcal{E},\mathcal{X}} $ which by the maximality assumption gives us that there exists $\tilde z\in \mathcal{Z}_1$ such that $e = S \tilde z$ and $ x = P \tilde z$. Putting it in \eqref{eqn:orthogonal-relation} and applying Assumption~\ref{eqn:symmetry-assumption} yields: 
    $$ 
    \forall \, z \in \mathcal{Z}_1, \qquad 
        0 = \langle \beta(z), \chi - \gamma(\tilde z)  \rangle_{\mathcal{U},\mathcal{Y}} - \langle \varepsilon - \beta( \tilde z) , \gamma(z) \rangle_{\mathcal{U},\mathcal{Y}} \, .
    $$
    Considering that $\begin{bmatrix}
        0 & -1 \\  1 & 0
    \end{bmatrix}$ is invertible and that $\begin{bmatrix}
        \gamma \\ \beta
    \end{bmatrix}$ has a dense range by Assumption~\ref{eqn:surjectivity-assumption}, yields that $\gamma(\tilde z) = \chi$ and $\beta(\tilde z) = \varepsilon$. This shows that $\begin{bmatrix} x & \chi & e & \varepsilon \end{bmatrix}^\top$ belongs to the image representation of the Lagrange structure, hence: 
    $$
    \begin{bmatrix}
        x & \chi & e & \varepsilon
    \end{bmatrix}^\top \in \mathcal{L} \, .
    $$
\end{proof}

\section{Variational derivatives and power balances} \label{apx:var-deriv}

\subsection{Variational derivatives of $H_W^{SL}$} \label{apx-subsec:wave-varder-lagrange}
\begin{lemma} The variational derivatives of $H_W^{SL}$ read:
    $$
        \delta_w H_W^{SL}= - \partial_x(E\partial_x w) \, , \qquad \delta_p H_W^{SL} = \frac{1}{\rho}p \, .
    $$
\end{lemma}
\begin{proof}
Let us firstly compute $\delta_w H_W^{SL}$, let $h \in \mathbb{R}$, $w \in H^1([a,b],\mathbb{R})$ and $f \in H^1_0([a,b],\mathbb{R})$, then:
    $$
        \begin{aligned}
            H_W^{SL}(w + hf,p) - H_W^{SL}(w,p) &= \frac{1}{2} \int_\Omega \partial_x(w + hf)\,  E \partial_x(w + hf) - \partial_x w\, E \partial_x w \, \d x \, ,\\
            & = \int_\Omega h \partial_x f\, E \partial_x w\, \d x + o(h) \, , \\
            & = - \int_\Omega h f \partial_x(E\partial_x w) \, \d x + o(h) \, , \\
            & = h \langle - f, \partial_x(E\partial_x w) \rangle_{L^2} + o(h) \, .
        \end{aligned}
    $$
Now let us compute $\delta_p H_W^{SL}$, let $h \in \mathbb{R}$ and $f \in L^2([a,b],\mathbb{R})$, then:
$$
    \begin{aligned}
        H_W^{SL}(w ,p+ hf) - H_W^{SL}(w,p) &= \frac{1}{2}\int_\Omega \frac{1}{\rho} ((p + hf)^2 - p) \, \d x \, , \\
        & = \int_\Omega hf \, \frac{1}{\rho}p \, \d x + o(h) \, , \\
        & = h\langle f, \frac{1}{\rho}p\rangle_{L^2} + o(h) \, .
    \end{aligned}
$$
Hence, we get:
$$
     \delta_w H_W^{SL}= - \partial_x(E\partial_x w) \, , \qquad \delta_p H_W^{SL} = \frac{1}{\rho}p \, .
$$
\end{proof}

\subsection{Variational derivatives of $H_{RM}^{SL}$} \label{apx-subsec:reissner-mindlin-varder-lagrange}
\begin{lemma}
    The Variational derivatives of $H_{RM}^{SL}$ read:
   $$
       \begin{bmatrix}
        \delta_w H_{RM}^{SL} \\
        \delta_{p_w} H_{RM}^{SL} \\
        \delta_{\bm{\phi}} H_{RM}^{SL} \\
        \delta_{\bm{p_\phi}} H_{RM}^{SL}
    \end{bmatrix} 
    =  \begin{bmatrix}
        - \diver ( kGh \, \grad( \cdot ) ) & 0 & - \diver ( kGh \, \cdot)  &  0\\
        0 & \frac{1}{\rho h} & 0 & 0\\
        kGh \, \grad(\cdot) & 0 &  -\Diver( \mathbb{D} \, \Grad(\cdot)) + kGh  & 0\\
        0 & 0 & 0 & \frac{12}{\rho h^3}
    \end{bmatrix} \begin{bmatrix}
         w \\
         p_w \\
         \bm{\phi} \\
         \bm{p_\phi}
    \end{bmatrix}.
   $$
\end{lemma}
\begin{proof}
Let us compute the 4 different variational derivatives using the Stokes' identities presented in~\ref{apx:stokes-identities}. Let us denote by $f \in H^2_0(\Omega)$ and $\bm{F} \in H^2_0(\Omega, \mathbb{R}^{2})$ a scalar and a vector valued test functions; moreover let us denote by $h\in \mathbb{R}$ a scalar.
    \begin{itemize}
        \item $\delta_w H_{RM}^{SL}$: 
        
        Denoting by $\Delta_w^h :=H_{RM}^{SL}(w + hf, p_w, \bm{\phi}, \bm{p_\phi}) - H_{RM}^{SL}(w, p_w, \bm{\phi}, \bm{p_\phi})$, we get: 
        $$ 
        \begin{aligned}
            \Delta_w^h  =& \frac{1}{2} \left \langle kGh \, \grad(w + hf) - \bm{\phi}, \grad(w + hf) - \bm{\phi} \right \rangle_{L^2}\\
                    & \qquad  -  \frac{1}{2} \left \langle kGh \, \grad(w) - \bm{\phi}, \grad(w) - \bm{\phi} \right \rangle_{L^2} \, , \\
                    =& \left \langle f, -  \diver\left( kGh \, (\grad(w) - \bm{\phi})\right) \right\rangle_{L^2} + o(h) \, .
        \end{aligned}
        $$
        \item $\delta_{p_w} H_{RM}^{SL}$: 
        
        Denoting by $\Delta_{p_w}^h := H_{RM}^{SL}(w , p_w + hf, \bm{\phi}, \bm{p_\phi}) - H_{RM}^{SL}(w, p_w, \bm{\phi}, \bm{p_\phi})$, we get:
        $$ 
        \begin{aligned}
            \Delta_{p_w}^h & = \frac{1}{2}\langle \frac{1}{\rho h} (p_w + hf), (p_w + hf) \rangle_{L^2} - \frac{1}{2} \langle \frac{1}{\rho h} p_w, p_w \rangle_{L^2} \, ,\\
            & = h \langle f, \frac{1}{\rho h} p_w \rangle_{L^2} + o(h) \, .
        \end{aligned}
        $$
        \item $\delta_{\bm{\phi}} H_{RM}^{SL}$: 
        
        Denoting by $\Delta_{\bm{\phi}}^h := H_{RM}^{SL}(w + hf, p_w, \bm{\phi} + h \bm{F}, \bm{p_\phi}) - H_{RM}^{SL}(w, p_w, \bm{\phi}, \bm{p_\phi})$, we get: 
        $$ 
        \begin{aligned}
            \Delta_{\bm{\phi}}^h  = & \frac{1}{2} \langle kGh \left(\grad(w) - (\bm{\phi} + h \bm{F})\right),  \grad(w) - (\bm{\phi} + h \bm{F}) \rangle_{L^2}  \\
                                    & \qquad + \frac{1}{2} \langle \mathbb{D} \, \Grad(\bm{\phi} + h \bm{F}), \Grad(\bm{\phi} + h\bm{F}) \rangle_{L^2}  \\
                                    & \qquad -\frac{1}{2} \langle kGh \, \grad(w) - (\bm{\phi}),  \grad(w) - (\bm{\phi}) \rangle_{L^2}  \\
                                    & \qquad - \frac{1}{2} \langle \mathbb{D} \, \Grad(\bm{\phi} ), \Grad(\bm{\phi} ) \rangle_{L^2} \, , \\
                                =& h \langle \bm{F}, kGh \, (\grad(w) - \bm{\phi} ) \rangle_{L^2} \\
                                    & \qquad +  h \langle \bm{F}, - \Diver(\mathbb{D}\,\Grad(\bm{\phi}))\rangle_{L^2} + o(h) \, .
        \end{aligned}
        $$
        \item $\delta_{\bm{p_\phi}} H_{RM}^{SL}$: 
        
        Denoting by $\Delta_{\bm{p_\phi}}^h :=  H_{RM}^{SL}(w + hf, p_w, \bm{\phi}, \bm{p_\phi}  + h \bm{F}) - H_{RM}^{SL}(w, p_w, \bm{\phi}, \bm{p_\phi}),$ we get:
        $$ 
        \begin{aligned}
           \Delta_{\bm{p_\phi}}^h = &\frac{1}{2} \langle \frac{12}{\rho h^3} (\bm{p_\phi} + h \bm{F}), (\bm{p_\phi} + h \bm{F})\rangle_{L^2}  + \frac{1}{2} \langle  \frac{12}{\rho h^3} \bm{p_\phi}, \bm{p_\phi}\rangle_{L^2} \, , \\
                                  = & h \langle \bm{F}, \frac{12}{\rho h^3} \bm{p_\phi}\rangle_{L^2} + o(h) \, .
        \end{aligned}
        $$
    \end{itemize}
\end{proof}

\section{Stokes' identities} \label{apx:stokes-identities}

Let us consider $\Omega \subset \mathbb{R}^3$ a domain and $\bm{n}$ denote the outer facing normal of the boundary $\partial \Omega$ of $\Omega$.

\subsection{$\Diver$ and $\Grad$}

\begin{theorem}
    Let $\bm{\psi} \in H^1(\Omega, \mathbb{R}^3), \; \tensor{\phi} \in H^1(\Omega, \mathbb{R}^3\otimes \mathbb{R}^3),$ with $\tensor{\psi} = \tensor{\psi}^\top,$ then:
    $$
        \int_\Omega \Grad(\bm{\psi}) : \tensor{\phi} \, \d x= - \int_\Omega \bm{\psi} \cdot \Diver(\tensor{\phi}) \, \d x + \int_{\partial \Omega} \bm{\psi}^\top \, \tensor{\phi} \, \bm{n} \, \d s \, .
    $$
\end{theorem}
\begin{proof}
    $$ 
    \begin{aligned}
        \int_\Omega \Grad(\bm{\psi}) : \tensor{\phi} & = \frac{1}{2} \int_\Omega \nabla \bm{\psi} : \tensor{\phi} + \nabla \bm{\psi}^\top : \tensor{\phi} \,\d x \, ,\\
        &= \int_\Omega \nabla \bm{\psi}: \frac{1}{2}( \tensor{\phi} + \tensor{\phi}^\top) \, \d x \, , \\
        &= \int_\Omega \sum_{ij} \frac{\partial \bm{\psi}_j}{\partial x_i} \, \tensor{\phi}_{ij} \, \d x \, ,\\
        & = -  \int_\Omega \sum_{ij} \bm{\psi}_j \, \frac{\partial \tensor{\phi}_{ij}}{\partial x_i} \d x+ \int_{\partial \Omega} \sum_{ij} \bm{\psi_j} \tensor{\phi}_{ij} \bm{n}_i \, \d s \, , \\
        & = - \int_\Omega \bm{\psi} \cdot \Diver(\tensor{\phi}) \, \d x+ \int_{\partial \Omega} \bm{\psi}^\top \tensor{\phi} \, \bm{n} \, \d s \, .
    \end{aligned} 
    $$
\end{proof}

\subsection{$\curl$ and $\curl$}

\begin{theorem}
    Let $\bm{\phi}$, $\bm{\psi} \in H^\curl(\Omega)$ be two vector-valued functions, then:
    $$
        \int_\Omega \bm{\phi} \cdot \curl(\bm \psi) \, \d x = \int_\Omega \curl(\bm{\phi}) \cdot \bm \psi \, \d x + \int_{\partial \Omega} (\bm{\psi} \times \bm{\phi}) \cdot \bm{n} \, \d s \, .
    $$
\end{theorem}
\begin{proof}
Let us first compute $\int_{\partial \Omega} (\bm{\psi} \times \bm{\phi}) \cdot \bm{n} \, \d s$ using Stokes' identity:
$$
\begin{aligned}
    \int_{\partial \Omega} (\bm{\psi} \times \bm{\phi}) \cdot \bm{n} \, \d s & = \int_\Omega \diver( \bm{\psi} \times \bm{\phi} ) \, \d x \, , \\
                                                                             & = \int_\Omega \curl(\bm{\psi}) \cdot \bm{\phi} - \bm{\psi} \cdot \curl(\bm{\phi}) \, \d x \, .
\end{aligned}
$$
The last line is obtained with classical vector calculus identities. Finally, rearranging terms on both sides gives the result.

\end{proof}

\end{document}